\documentclass[11pt]{article}

\usepackage[utf8]{inputenc}
\usepackage{geometry}
\usepackage{amsthm}

\usepackage{booktabs}
\usepackage{graphicx}   
\usepackage{pdflscape}
\newtheorem{theorem}{Theorem}[section]
\newtheorem{lemma}[theorem]{Lemma}

\theoremstyle{definition}

\theoremstyle{remark}

\usepackage{amsmath,amsfonts,amssymb,amsthm,mathtools}
\usepackage{mathrsfs}
\usepackage{dsfont}

\usepackage{graphicx}
\usepackage{subcaption}   
\usepackage{booktabs}
\usepackage{multirow}
\usepackage{array}

\usepackage{xcolor}
\usepackage{float}
\usepackage{placeins}

\usepackage{algorithm}
\usepackage{supertabular}

\usepackage[numbers]{natbib}

\usepackage{pgfplots}
\pgfplotsset{compat=1.7}
\usetikzlibrary{intersections}

\usepackage{fancyhdr}
\usepackage{etoolbox}
\usepackage{accents}

\usepackage{upgreek}

\usepackage[colorlinks=true,citecolor=blue,linkcolor=blue]{hyperref}

\begin{document}
\title{{\em A High-Order Conformal FEM for Multidimensional Nonlinear Collisional Breakage Equations: Analysis and Computation}}

\author{
    Arushi\thanks{Department of Mathematics \& Computing, 
    Dr B R Ambedkar National Institute of Technology Jalandhar, Punjab, India. 
    Email:  arushi.mc.25@nitj.ac.in}
    \;\; and \;\;
    Naresh Kumar\thanks{Department of Mathematics \& Computing, 
    Dr B R Ambedkar National Institute of Technology Jalandhar, Punjab, India. 
    Corresponding author: nareshk@nitj.ac.in}
}
\date{}
\maketitle

\begin{abstract}
Particle breakage due to collisional interactions plays a vital role in the development of several phenomena in science and engineering. The nonlinear collisional breakage equations (NCBEs) are a significant set of equations in this context. Solving the NCBE is computationally challenging due to its nonlinearity, high dimensionality, and complex kernel interactions. Solving NCBE problems is more complex in two- and three-dimensional problems. In these problems, it is more challenging to evaluate multidimensional moments and integrals, maintain solution stability, and achieve computational efficiency. Despite the importance of the NCBE in science and engineering, the development of efficient numerical methods for solving it in two- and three-dimensional problems has not been adequately explored. In this work, we have introduced a new framework for solving the NCBE across multiple dimensions using the conformal finite element method (FEM). To the best of our knowledge, this is the first work to solve the NCBE using the conformal FEM. The new framework employs high-order Lagrange elements in conjunction with the BDF2 scheme for time discretization. The present method preserves the important physical quantities such as the total count and hypervolume of the population particles. Convergence results for error estimates have also been derived for both semidiscrete and fully discrete schemes. Numerical experiments have been carried out for one-, two-, and three-dimensional problems. The numerical experiments have shown that the proposed method achieved high accuracy, optimal convergence rates, and computational efficiency.
\end{abstract}

\noindent\textbf{Keywords:} Nonlinear collisional breakage equation; Multidimensional fragmentation; Conformal finite element method;  BDF2 time discretization; Error analysis.

{\em AMS Subject Classifications(2000)}. 65N15, 65N30

\section{Introduction}\label{sec:1}
\subsection{Modelling background}
The breakage process is a kinetic, irreversible process. Breakage occurs due to external forces or the interaction of two or more particles, which ultimately changes particle shape and divides them into two or more fragments. Breakage is a significant process for changing a particle's properties, such as its size, composition, and porosity. The phenomenon of breakage is a fundamental process across several scientific and engineering fields, including crystallization, nanotechnology, emulsions, and depolymerization \cite{leong2023comparative}. In several industrial and biological systems, breakage is not random but highly deliberate. Breakage is a very important phenomenon in grinding and milling, as well as in the breakage of protein filaments. \cite{tournus2021insights}, \cite{wang2021multiscale}. On a broader level, breakage can be divided into two types: linear and nonlinear. The process of linear breakage occurs due to internal stresses and external conditions, such as external forces and temperature, etc., on a particle. Nonlinear breakage is generally due to the interaction and collision of two or more particles. \cite{hussain2025hybridized}, \cite{saha2023rate}. The major difference between these two phenomena of breakage is that the fragments formed due to the process of breakage are smaller than the parent particles. This is depicted in Figure \ref{fig:breakage}. On the other hand, nonlinear breakage can facilitate mass transfer in the collision, which has the potential to produce fragments that can be larger than the original particles \cite{planchette2017colliding}. This has made nonlinear breakage more versatile than linear breakage, making it better suited to model complex physical processes and transforming the field. One model of nonlinear breakage is binary collisional breakage, which occurs when two particles collide instantaneously and inelastically \cite{laurenccot2001discrete}. This results in the fragmentation of one or two particles. Nonlinear breakage has transformed the understanding of many processes in different fields, including planet formation, fluid bed, radiobiology, raindrop formation, asteroid distribution, and communication systems, among others \cite{brilliantov2015size, fries2013collision, kiefer1986model, mcfarquhar2004new, piotrowski1953collisions, fuerstenau2004linear}. Despite many developments in controlling fragmentation processes, including dynamical forecasting, numerical modeling, formal asymptotic analysis, and solution techniques, the mathematical literature on collisional equations of breakage is sparse. This has formed the basis for this research, which aims to develop advanced numerical techniques for nonlinear collisional breakage. The main objective of this research is to develop a powerful, efficient computational model that handles diverse cases appropriately and performs well in real time.
\begin{figure}[htbp]
    \centering
    \includegraphics[width=0.6\textwidth]{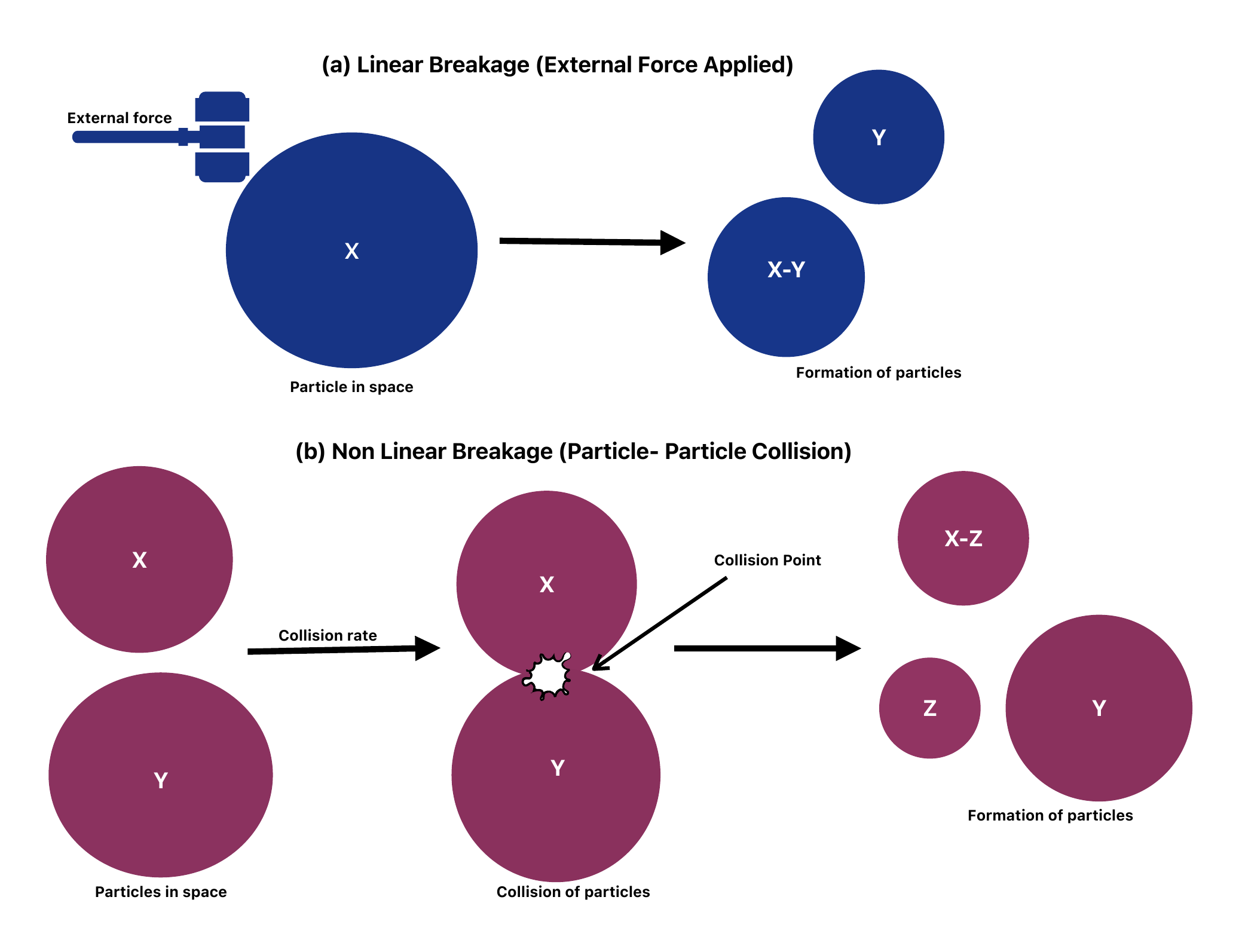} 
     \caption{ Comparison of particle breakage mechanisms}
    \label{fig:breakage}
\end{figure}
\vspace{-0.32cm}
\subsection{Problem Formulation }\label{sec:sec1}
\vspace{-0.3cm}

From a mathematical standpoint, the nonlinear collisional breakage equation (NCBE) describes the breakage of particles into smaller fragments due to kinetic interactions. The model takes the form of a nonlinear integro-partial differential equation. The continuous collision breakage equation (CBE) is the focus of the present study. Cheng and Redner developed the integro-partial differential framework \cite{cheng1988scaling}. Let $\mathbf{x} := (x_1,x_2,.....,x_d) \in \mathbb{R}_+^d$ 
denote the property attributes vector, and population of particles described by a number density function $u(\mathbf{x},t) \ge 0$, and $t \ge 0$ represents time.
\begin{equation}
\partial_t u(\mathbf{x},t)
=\int_{0}^{\infty}\int_{\mathbf{x}}^{\infty} 
\beta(\mathbf{x},\mathbf{y},\mathbf{z})\Gamma(\mathbf{y},\mathbf{z})u(\mathbf{y},t)u(\mathbf{z},t)\,d\mathbf{y}\,d\mathbf{z} - \int_{0}^{\infty} \Gamma(\mathbf{x},\mathbf{y})u(\mathbf{x},t)u(\mathbf{y},t) \, d\mathbf{y} ,
\label{eq:model}
\end{equation}
subject to the initial condition $u(\mathbf{x}, 0) = u_0(\mathbf{x})\ge0$.

The kernels can determine the behaviour of the system:
\begin{itemize}
\item \textbf{Collision Kernel ($\Gamma(\mathbf{y},\mathbf{z})$):} It represents the rate of successful collision-induced fragmentation events between particles with characterization $\mathbf{y}$ and $\mathbf{z}$. Also, there exists symmetry in the kernel function, that is $\Gamma(\mathbf{y},\mathbf{z})$ = $\Gamma(\mathbf{z},\mathbf{y}).$
\item \textbf{Breakage Distribution Function ($\beta(\mathbf{x},\mathbf{y},\mathbf{z})$):} It represents the rate at which particles of characteristics $\mathbf{x}$ are generated from a particle of characteristics $\mathbf{y}$ as a result of its collision with a particle of characteristics $\mathbf{z}$.
\end{itemize}
\vspace{-0.3cm}

To ensure physical consistency, the kernel satisfies the following structural properties:
\begin{enumerate} 
\vspace{-0.3cm}

\item [(i)] \textbf{No oversize Fragments:} Attributes of the resultant fragment cannot exceed the total attributes of the colliding parents:
\begin{equation}
\beta(\mathbf{x},\mathbf{y},\mathbf{z}) = 0 \quad \text{if} \quad \exists , i \in {1, \dots, d} : x_i > y_i .
\end{equation}
\item [(ii)]  \textbf{Multiplicity of fragments:}  
The average number of fragments produced per breakage is:
\begin{equation}
\int_{0}^{\mathbf{y}} \beta(\mathbf{x}, \mathbf{y}, \mathbf{z})  d\mathbf{x} = \nu(\mathbf{y},\mathbf{z}) \quad \text{where} \quad 2 \le \nu(\mathbf{y},\mathbf{z}) < \infty.
\label{eq:multiplicity}
\end{equation}
\item [(iii)]  \textbf{Hypervolume conservation:}  The total hypervolume of all the fragments after fragmentation is equal to the hypervolume of the mother particle
before fragmentation:
\begin{equation}
\int_{0}^{\mathbf{y}} \left( \prod_{i=1}^{d} x_i \right) \beta(\mathbf{x}, \mathbf{y}, \mathbf{z})  d\mathbf{x} = \prod_{i=1}^{d} y_i .
\label{eq:mass}
\end{equation}
\end{enumerate}
\textbf{Moment Dynamics:}
Apart from the number density function itself, some integral properties, such as the moments, are helpful for understanding the distribution's statistical behavior, like total mass, total number, composition contents, etc. For a distribution function $u(\mathbf{x}, t)$, if $k = \sum_{i=1}^{d} k_i,$ the $k$-th moment is defined as
\[
\mathcal{M}_{k_1,k_2,\dots,k_d}(t)
= \int_{0}^{\infty}
\left( \prod_{i=1}^{d} x_i^{\,k_i} \right)
u(\mathbf{x},t)\, d\mathbf{x},
\quad \text{where } k_i \in \mathbb{N}, \; i = 1,2,\dots,d.
\]
\paragraph{Total Number Density ($\mathcal{M}_{0,0,.....,0}$):}
The zeroth moment ($k = 0$) represents the total number of particles in the system. Differentiating with respect to time yields:
\[ \frac{d\mathcal{M}_{0,0,.....,0}}{dt} = \int_0^{\infty}\int_0^{\infty} (\nu(\mathbf{y},\mathbf{z})-1)\Gamma(\mathbf{y},\mathbf{z})u(\mathbf{y},t)u(\mathbf{z},t) \, d\mathbf{y} \, d\mathbf{z}. \]
As already given $\nu(\mathbf{y},\mathbf{z}) \ge 2$ and $\Gamma(\mathbf{y},\mathbf{z})\ge 0$, it follows that $d\mathcal{M}_{0,0,.....,0}/dt > 0$, which is the rate of change in total number of particles. The first moments (1 only at the i-th index) describe the total property content corresponding to the i-th property.

\vspace{-0.3cm}

\paragraph{Total Hypervolume :} $\mathcal{M}_{1,1,.....1}$ represents the total hypervolume.
The concept of hypervolume is employed to characterize the aggregate measure of internal properties in multivariate objects. This property remains invariant under collisional fragmentation. Furthermore, it is easy to obtain:
\[ \frac{d\mathcal{M}_{1,1,...,1}}{dt} = 0. \]

\vspace{-0.4cm}
\subsection{State of the Art}
\vspace{-0.25cm}

The mathematical modeling of particle systems has been extensively studied in the context of aggregation, linear breakage, and coupled aggregation--breakage processes within population dynamics \cite{dubovskii1992exact,attarakih2009solution}. However, the nonlinear collisional breakage equation (CBE) introduces additional layers of complexity due to its intrinsic nonlinearity and coupling structure. Although its formulation shares similarities with linear models \cite{cheng1990kinetics,giri2021existence,ernst2007nonlinear}, obtaining general analytical solutions remains a significant challenge. To date, exact solutions are restricted to simplified scenarios, such as constant or specific collision kernels \cite{kostoglou2000study,ziff1985kinetics}, and a limited class of self-similar solutions. These analytical limitations have motivated the development of numerical and semi-analytical approaches. Over the past two decades, the modeling of collisional breakage has undergone substantial theoretical and computational advancements. Early contributions by Laurençot and Wrzosek \cite{laurenccot2001discrete} established existence, uniqueness, and conservation properties for discrete coagulation equations with collisional breakage, albeit under linearized breakage assumptions. Subsequent work by Vigil et al. \cite{vigil2006destructive} introduced population balance models for destructive aggregation, further extending the modeling framework. Later, Barik and Giri \cite{barik2020global} and Paul and Kumar \cite{paul2018existence} provided rigorous results on the existence and uniqueness of solutions for nonlinear coagulation and pure collisional breakage equations under suitable conditions. From a numerical perspective, a wide range of methods has been proposed to approximate solutions and moment dynamics of NCBE. These include the finite volume method (FVM) \cite{das2020approximate,paul2023moments}, sectional-volume average methods \cite{kushwah2025convergence}, Monte Carlo simulations \cite{das2020approximate}, fixed pivot techniques (FPT), and semi-analytical approaches such as variational iteration and Adomian decomposition methods \cite{arora2023comparison,yadav2023homotopy,keshav2025new,arora2026accurate}. More recently, Lombart et al. \cite{lombart2022fragmentation} employed a discontinuous Galerkin (DG) scheme for univariate collisional breakage equations, demonstrating reduced numerical diffusion for larger particle sizes.
 Despite these advances, several limitations persist. Monte Carlo methods are computationally intensive because they require a large number of samples to achieve accuracy. FVM performs well on structured grids but deteriorates on coarse or irregular meshes. Fixed pivot techniques are limited to first-order convergence and may fail on non-uniform grids. Semi-analytical methods often incur high computational costs, while DG formulations may produce nonphysical negative solutions without additional stabilization mechanisms \cite{lombart2024general}. Notably, the application of conformal FEM to nonlinear collisional breakage equations remains largely unexplored, with existing FEM-based approaches primarily confined to discontinuous formulations. This highlights a critical gap in the literature and underscores the need for robust, high-order, and structure-preserving FEM frameworks capable of handling multidimensional nonlinear collisional breakage problems.

\vspace{-0.3cm}
\subsection{Contributions of the Present Work}
\vspace{-0.25cm}

Although most existing numerical and analytical work on the numerical solution of the collisional breakage equations is restricted to linear and one-dimensional nonlinear cases, the problems encountered in practice are inherently multidimensional. As such, the development of an efficient numerical methodology is unavoidable. In this work, the high-order conformal FEM is developed for the numerical solution of the nonlinear collisional breakage equation, based on rigorous variational analysis using $C^0$ finite element spaces. It is worth noting that this is the first work to provide a systematic investigation of the nonlinear collisional breakage equation in one-, two-, and three-dimensional cases. It is worth noting that the main contributions of the present work can be summarized as follows:
\vspace{-0.3cm}
\begin{enumerate}
\item[(i)] \textbf{High-order FEM framework with rigorous analysis:} 
In this work, a high-order conformal FEM framework is developed for the numerical solution of the nonlinear collisional breakage equation, based on rigorous analysis of the problem. It is worth noting that the high-order conformal FEM is based on the Lagrange basis functions and is supported by the theoretical analysis.

\item[(ii)] \textbf{Structure-preserving and physically consistent formulation:} 
The numerical method developed is intrinsically structure-preserving and physically consistent. The numerical method preserves the important physical quantities such as the total count and hypervolume of the population particles.

\item[(iii)] \textbf{Robustness of the method for a wide class of problems and initial conditions:} 
The numerical method developed is validated for a wide class of problems and initial conditions. The numerical method is validated for a wide range of analytically tractable and physically relevant kernels, including product and polymerization kernels. The numerical method is also validated for a wide range of initial conditions, including Dirac delta and exponential initial conditions.

\item[(iv)] \textbf{Comprehensive multidimensional validation of the method with superior performance:} 
The numerical method developed is validated across one, two, and three-dimensional problems. The numerical method is validated through extensive numerical experiments by comparing the numerical solutions obtained with the method developed with those from existing numerical methods and exact solutions. The numerical solutions obtained using the developed method are more accurate than those from existing methods.
\end{enumerate}
\noindent{\bf Organization of the paper}: The remainder of this paper is structured as follows:  In Section~\ref{sec:2}, we present the truncation of the computational domain and derive the corresponding weak formulation within the Finite Element Method (FEM) framework. In Section~\ref{sec:3},
The constructed semi-discrete FEM scheme is verified to ensure conservation of the total amount of individual physical properties, such as hypervolume and number, and we provide a rigorous analysis of its stability and error estimates. This analysis is extended to the fully-discrete algorithm in Section~\ref{sec:4}, where we present corresponding stability and error bounds. Section~\ref{sec:5} is dedicated to the numerical validation of the proposed scheme, applying it to 1D, 2D, and 3D nonlinear collision equations. Finally, in section \ref{sec:6}, we have reported  the conclusion of the paper.

\vspace{-0.4cm}
\section{Finite Element Approximation} \label{sec:2}
\vspace{-0.3cm}

In this section, we construct a Galerkin finite element approximation of the truncated nonlinear collision-fragmentation model on a bounded computational domain.

\vspace{-0.35cm}

\subsection{Continuous weak formulation}
\vspace{-0.35cm}
The collisional fragmentation model is defined on the unbounded domain $\mathbb{R}_+^d$. However, for the purpose of numerical approximation, we must confine computations to a bounded computational domain $\mathcal{D} := (0, x_{1\max}]\times (0, x_{2\max}]\times .....\times (0, x_{d\max}]$, where $0 < \mathbf{x}_{\max} < \infty  $, over a finite time interval $t \in [0, T]$. 
Under this restriction, the evolution of the particle number density function $u(\mathbf{x},t)$ is governed by the truncated nonlinear integro partial differential equation:
\begin{equation}
\partial_t u(\mathbf{x},t)
=\int_\mathcal{D}\int_{\mathbf{x}}^{\mathbf{x}_{\max} } 
\beta(\mathbf{x},\mathbf{y},\mathbf{z})\Gamma(\mathbf{y},\mathbf{z})u(\mathbf{y},t)u(\mathbf{z},t)\,d\mathbf{y}\,d\mathbf{z} - \int_\mathcal{D} \Gamma(\mathbf{x},\mathbf{y})u(\mathbf{x},t)u(\mathbf{y},t) \, d\mathbf{y} ,
\label{truncated_model}
\end{equation}
subject to the initial condition:
\[ u(\mathbf{x},0) = u_0(\mathbf{x}), \quad \mathbf{x} \in \mathcal{D} \]
\paragraph{Well-Posedness:}
To ensure the existence and uniqueness of solutions for the truncated model \eqref{truncated_model}, we impose the following assumptions:
\begin{enumerate}
    \item  The kernel $\Gamma(\mathbf{x},\mathbf{y})$ is continuous and nonnegative on the domain , satisfying $0 \le \Gamma(\mathbf{x},\mathbf{y}) \le C_0 < \infty$ for all $(\mathbf{x},\mathbf{y}) \in \mathcal{D}$.
    
    \item  The breakage distribution function  $\beta(\mathbf{x},\mathbf{y},\mathbf{z})$ is continuous, nonnegative, and bounded such that  $\beta(\mathbf{x},\mathbf{y},\mathbf{z}) \le b_0$.
\end{enumerate}
\paragraph{Weak formulation:}  
Find $u:(0,T]\to V:=H^1_0( \mathcal{D})$ such that
\begin{equation}
(\partial_tu,\phi)+\mathcal{A}(u,u;\phi)=\mathcal{B}(u,u;\phi),
\quad \forall \phi \in V, t>0 ,
\label{eq:exactweak}
\end{equation}
with $u(0)=u_0$, where
\[
(\partial_tu,\phi):=\int_\mathcal{D} \partial_tu(\mathbf{x},t)\phi(\mathbf{x})d\mathbf{x},
\;\;
\mathcal{A}(u,u;\phi):=\int_\mathcal{D}\phi(\mathbf{x})u(\mathbf{x},t)
\left(\int_\mathcal{D}\Gamma (\mathbf{x},\mathbf{y})u(\mathbf{y},t)d\mathbf{y}\right)d\mathbf{x},
\]

\[
\mathcal{B}(u,u;\phi):=
\int_\mathcal{D}\phi(\mathbf{x})
\left(
\int_\mathbf{x}^{\mathbf{x}_{\max}}\int_\mathcal{D}
\beta(\mathbf{x},\mathbf{y},\mathbf{z})\Gamma(\mathbf{y},\mathbf{z})u(\mathbf{y},t)u(\mathbf{z},t)d\mathbf{y}\,d\mathbf{z}
\right)d\mathbf{x} .
\]
\subsection{Mesh Discretization}
\paragraph{\textbf{One-Dimensional Discretization:}}
We begin by partitioning the truncated domain $\mathcal{D}=(0,x_{\max})$ into $N$ non-overlapping subintervals 
$
I_i=(x_{i-1},x_i), \qquad i=1,2,\ldots,N,
$ with boundary points $x_0=0$ and $x_N=x_{\max}$. The size of each element is denoted by $h_i := x_i - x_{i-1}$, and we define the global mesh parameter as $h = \max_{1 \le i \le N} h_i$. For a given polynomial degree $r \geq 1$, we construct the conforming finite element space
\[
V_h := \left\{ v \in C^0(\overline{\mathcal{D}}) \; : \; 
v|_{I_i} \in \mathscr{P}_r(I_i), \quad i=1,2,\ldots,N \right\},
\]
which has dimension $\dim(V_h) = rN + 1$. This space admits a global Lagrange basis representation
\[
V_h = \mathrm{span}\{\phi_0(x), \phi_1(x), \ldots, \phi_{rN}(x)\},
\]
where the basis functions satisfy the nodal interpolation property
\begin{equation}
\phi_i(x_j) = \delta_{ij}, \qquad i,j = 0,1,\ldots,rN,
\end{equation}
with $\{x_j\}$ denoting the global nodal coordinates.

To construct these basis functions, we first introduce local coordinates $\xi \in [0,1]$ on a reference element $\widehat{I}=[0,1]$. The local Lagrange basis functions of degree $r$ are defined as
\begin{equation}
\psi_k(\xi) = \prod_{\substack{0 \le m \le r \\ m \neq k}} 
\frac{\xi - \xi_m}{\xi_k - \xi_m}, 
\qquad k=0,1,\ldots,r,
\end{equation}
where the reference nodes are given by $\xi_k = k/r$. These local basis functions are then mapped onto each physical element $I_i$ via an affine transformation, resulting in globally continuous shape functions.

\vspace{-0.4cm}

\paragraph{\textbf{Extension to Multidimensional Domains: }}
The construction above naturally extends to higher-dimensional settings. For instance, consider the two-dimensional truncated domain
$
\mathcal{D} := (0, x_{1\max}] \times (0, x_{2\max}].
$
Let $\mathscr{T}_h = \{{\cal K}\}$ be a conforming, shape-regular triangulation of $\mathcal{D}$, consisting of non-overlapping triangular elements such that
$
\overline{\mathcal{D}} = \bigcup_{{\cal K} \in \mathscr{T}_h} \overline{{\cal K}}.
$
The conformity of the mesh ensures that the intersection of two distinct elements ${\cal K}_1$ and ${\cal K}_2$ is empty, a common vertex, or a shared edge.
For each element ${\cal K} \in \mathscr{T}_h$, we define its diameter by $\delta_{\cal K} := \mathrm{diam}({\cal K})$, and the global mesh size by
$
h := \max_{{\cal K} \in \mathscr{T}_h} \delta_{\cal K}.
$
We assume that the family of meshes $\{\mathscr{T}_h\}_{h>0}$ satisfies the standard shape-regularity condition: there exist constants $\gamma_1, \gamma_2 > 0$, independent of $h$, such that
\[
\gamma_1 h \le \delta_{\cal K} \le \gamma_2 h, 
\qquad \forall {\cal K} \in \mathscr{T}_h.
\]
This guarantees uniform control over element distortion and ensures the stability and convergence of the numerical method. Finally, the mesh refinement process corresponds to the asymptotic limit $h \to 0$, leading to increasingly accurate spatial approximations.

\medskip
\noindent\textbf{Finite element approximation space:}
Associated with the triangulation $\mathscr{T}_h$, we introduce the finite-dimensional space
\[
V_h := \left\{ v_h \in C^0(\overline{\mathcal{D}}): 
v_h|_{\mathcal{K}} \in \mathscr{P}_r(\mathcal{K}), \ \forall \mathcal{K} \in \mathscr{T}_h \right\},
\]
where $ \mathscr{P}_r(\cal{K})$ denotes the space of polynomials of degree at most $r$ on the element $\cal{K}$.
Let $\{ \phi_i(\mathbf{x}) \}_{i=1}^{M}$ be the nodal basis of $V_h$, associated with the mesh nodes $\{N_i\}_{i=1}^{M}$. These basis functions satisfy $\phi_i(N_j) = \delta_{ij}, \quad 1 \le i,j \le M, $ and each $\phi_i$ has compact support restricted to the patch of elements sharing the node $N_i$.

\vspace{-0.4cm}
\section{Semidiscrete Galerkin approximation}\label{sec:3} \vspace{-0.3cm}

Let $V_h \subset V$ denote a finite-dimensional subspace spanned by basis functions $\{\varphi_i\}_{i=1}^N$. We approximate the exact solution $u(\mathbf{x},t)$ by the finite-dimensional expansion
$
u_h(\mathbf{x},t)=\sum_{i=1}^{N}\alpha_i(t)\varphi_i(\mathbf{x}),
$
where $\alpha_i(t)$ are unknown time-dependent coefficients and 
$
\boldsymbol{\alpha}(t) = (\alpha_1(t),\dots,\alpha_M(t))^{\top}
$ denotes the coefficient vector.
The semidiscrete formulation reads: find $u_h(t)\in V_h$ such that
\begin{equation}
(\partial_tu_h,\phi_h)+\mathcal{A}(u_h,u_h;\phi_h)=\mathcal{B}(u_h,u_h;\phi_h),
\qquad \forall \phi_h\in V_h .
\label{eq:weak}
\end{equation}
subject to the initial condition
$
  u_h(0) = \Pi_h u_0,
$
where $(\cdot,\cdot)$ denotes the $L^2(\mathcal{D})$ inner product and 
$\Pi_h$ is an appropriate projection onto $V_h$.

\vspace{-0.3cm}

\subsection{Algebraic Structure of the Semidiscrete System}
\vspace{-0.3cm}

Substituting the finite element approximation  $u_h(\mathbf{x},t)=\sum_{i=1}^{M}\alpha_i(t)\,\varphi_i(\mathbf{x}),
$ into the weak formulation and choosing the test functions as $\phi_h=\varphi_j$, 
for $j=1,\dots,M$, we obtain the following system of nonlinear ordinary differential equations:
\begin{equation}
\sum_{i=1}^{M} (\varphi_i,\varphi_j)\,\frac{d\alpha_i}{dt}
+
\sum_{i=1}^{M} \mathcal{A}(\varphi_i,\varphi_i,\varphi_j)\,\alpha_i^2
=
\sum_{i=1}^{M} \mathcal{B}(\varphi_i,\varphi_i,\varphi_j)\,\alpha_i^2,
\quad j=1,\dots,M.
\label{semi3}
\end{equation}
The above system can be expressed in a compact and insightful matrix form:
\[
\mathbf{M}\,\boldsymbol{\alpha}'(t) 
+ \mathbf{A}\,\boldsymbol{\alpha}^{\odot 2}(t)
= 
\mathbf{B}\,\boldsymbol{\alpha}^{\odot 2}(t),
\]
or equivalently,
\[
\boldsymbol{\alpha}'(t)
=
\mathbf{M}^{-1}(\mathbf{B}-\mathbf{A})\,\boldsymbol{\alpha}^{\odot 2}(t),
\]
where $\boldsymbol{\alpha}(t)=(\alpha_1(t),\dots,\alpha_M(t))^{\top}$ and 
$\boldsymbol{\alpha}^{\odot 2}$ denotes the component-wise (Hadamard) square. 

The matrices appearing in the system naturally encode the underlying physics and interactions:
\begin{itemize}
\item \textbf{Mass matrix:}
\[
M_{ij}=\int_{\mathcal{D}}\varphi_i(\mathbf{x})\,\varphi_j(\mathbf{x})\,d\mathbf{x},
\]
which captures the inner-product structure of the finite element space.
\item \textbf{Aggregation operator:}
\[
A_{ij}=
\int_{\mathcal{D}}
\varphi_j(\mathbf{x})\,\varphi_i(\mathbf{x})
\left(
\int_{\mathcal{D}}\Gamma(\mathbf{x},\mathbf{y})\,\varphi_i(\mathbf{y})\,d\mathbf{y}
\right)
d\mathbf{x},
\]
representing nonlinear interaction effects driven by the kernel $\Gamma$.
\item \textbf{Breakage operator:}
\[
B_{ij}=
\int_{\mathcal{D}}\varphi_j(\mathbf{x})
\left(
\int_{\mathbf{x}}^\mathbf{{x_{\max}}}
\int_{\mathcal{D}}
\beta(\mathbf{x},\mathbf{y},\mathbf{z})\,\Gamma(\mathbf{y},\mathbf{z})\,
\varphi_i(\mathbf{y})\,\varphi_i(\mathbf{z})
\,d\mathbf{y}\,d\mathbf{z}
\right)
d\mathbf{x},
\]
which accounts for redistribution phenomena arising from fragmentation processes.
\end{itemize}
\paragraph{\textbf{Bounds of the Nonlinear Operators:}}
To ensure well-posedness of the semi-discrete formulation, we first establish uniform bounds for the nonlinear operators 
$\mathcal{A}(\cdot,\cdot;\cdot)$ and $\mathcal{B}(\cdot,\cdot;\cdot)$.

\noindent
\textbf{Estimate for the aggregation operator $\mathcal{A}(u,u;\phi)$.}
Assume that the collision kernel $\Gamma(\mathbf{x},\mathbf{y})$ is uniformly bounded, i.e.,
$
|\Gamma(\mathbf{x},\mathbf{y})| \le C_0, 
\;\; \forall\, (\mathbf{x},\mathbf{y}) \in \mathcal{D}^d.
$
Then, by applying the Cauchy--Schwarz inequality, we obtain
\[
|\mathcal{A}(u,u;\phi)| 
\le 
C_0 \, |\mathcal{D}|^{1/2} \, \|\phi\| \, \|u\|^2,
\;\; \forall\, \phi,
\]
where $|\mathcal{D}| = \prod_{i=1}^{d} x_{i,\max}$ denotes the measure of the truncated domain.

\noindent
\textbf{Estimate for the breakage operator $\mathcal{B}(u,u;\phi)$.}
Similarly, suppose that the breakage function and collision kernel satisfy
$
|\beta(\mathbf{x},\mathbf{y},\mathbf{z})| \le b_0,
\;\; 
|\Gamma(\mathbf{x},\mathbf{y})| \le C_0.
$
Using Hölder's inequality together with the Cauchy--Schwarz inequality, we deduce the bound
\[
|\mathcal{B}(u,u;\phi)| 
\le 
C_0\, b_0 \, |\mathcal{D}|^{3/2} \, \|\phi\| \, \|u\|^2,
\quad \forall\, \phi.
\]
\paragraph{\textbf{Existence and Uniqueness of the Semidiscrete Solution.}}
We now turn to the solvability of the resulting semidiscrete system. Recall that the algebraic system can be written as
\[
\boldsymbol{\alpha}'(t)
=
\mathbf{M}^{-1}(\mathbf{B}-\mathbf{A})\,\boldsymbol{\alpha}^{\odot 2}(t),
\]
where $\mathbf{M}$ is the mass matrix. Since $\mathbf{M}$ is symmetric and positive definite, it is invertible. Consequently, the system reduces to a finite-dimensional nonlinear ODE with continuous (indeed, locally Lipschitz) right-hand side.
By the classical Cauchy--Lipschitz (Picard--Lindelöf) theorem, for any initial condition 
$\boldsymbol{\alpha}(0)=\boldsymbol{\alpha}^0 \in \mathbb{R}^M$, 
there exists a unique solution
$
\boldsymbol{\alpha}(t) \in C^{1}([0,T];\mathbb{R}^{M}).
$
\vspace{-0.34cm}
\subsection{Stability Analysis and Mass Conservation}
\vspace{-0.24cm}
We establish the stability of the semidiscrete finite element approximation in the following lemma.

\begin{lemma} [Stability of the Semidiscrete Scheme] Let $u_h(t)\in V_h$ be the semidiscrete Galerkin solution satisfying \eqref{eq:weak} with initial data $u_h(0)\in V_h$. Then the semidiscrete solution satisfies the stability estimate
\[
\|u_h(t)\|\le \frac{\|u_h(0)\|}{1-Kt\|u_h(0)\|}, with \quad t < \frac{1}{K\|u_h(0)\|},
\]
where $K:=C_0 b_0 (\prod_{i=1}^dx_{i\max})^{3/2}+C_0(\prod_{i=1}^dx_{i\max})^{1/2}$
\end{lemma}
\begin{proof}
We test the semidiscrete formulation \eqref{eq:weak} with $\phi_h = u_h$, which yields
\[
(\partial_t u_h, u_h) + \mathcal{A}(u_h,u_h,u_h) = \mathcal{B}(u_h,u_h,u_h).
\]
Observing that
$
(\partial_t u_h, u_h) = \frac{1}{2}\frac{d}{dt}\|u_h\|^2,
$
we obtain
\[
\frac{1}{2}\frac{d}{dt}\|u_h\|^2 
= \mathcal{B}(u_h,u_h,u_h) - \mathcal{A}(u_h,u_h,u_h).
\]
Using the boundedness estimates for the nonlinear operators $\mathcal{A}(\cdot,\cdot;\cdot)$ and $\mathcal{B}(\cdot,\cdot;\cdot)$, we deduce
\[
|\mathcal{A}(u_h,u_h,u_h)| \le C_0 |\mathcal{D}|^{1/2}\|u_h\|^3,
\qquad
|\mathcal{B}(u_h,u_h,u_h)| \le C_0 b_0 |\mathcal{D}|^{3/2}\|u_h\|^3,
\]
where $|\mathcal{D}| = \prod_{i=1}^{d} x_{i,\max}$.

Combining these estimates, we arrive at the differential inequality
\[
\frac{1}{2}\frac{d}{dt}\|u_h\|^2 
\le 
K \|u_h\|^3,
\]
where the constant $K>0$ is defined by
$
K: = C_0 b_0 |\mathcal{D}|^{3/2} + C_0 |\mathcal{D}|^{1/2}.
$
Setting $y(t) := \|u_h(t)\|$, we obtain
\[
y(t)\,y'(t) \le K y^3(t),
\quad \text{which implies} \quad
y'(t) \le K y^2(t).
\]
This yields the differential inequality
\[
\frac{d}{dt}\big(y^{-1}(t)\big) \ge -K.
\]
Integrating over $[0,t]$, we obtain
\[
y^{-1}(t) \ge y^{-1}(0) - Kt.
\]
Rewriting the above inequality gives the explicit bound
\[
\|u_h(t)\| \le 
\frac{\|u_h(0)\|}{1 - K t \|u_h(0)\|},
\;\; \text{for } t < \frac{1}{K\|u_h(0)\|}.
\]
This establishes a {finite-time a priori bound} for the semidiscrete solution, highlighting the nonlinear growth behavior governed by the aggregation and breakage mechanisms.
\end{proof}
\vspace{-0.4cm}

\subsection*{Hypervolume Conservation}
\vspace{-0.2cm}

To establish hypervolume conservation, we choose
\(\phi_j(\mathbf{x})=\prod_{i=1}^dx_i\) in equation \eqref{semi3}, and we obtain
\[
\sum_{i=0}^{N} \frac{d\alpha_i}{dt}(\varphi_i,\prod_{i=1}^dx_i)
+ \sum_{i=0}^{N} \alpha_i^2\,{\cal A}(\varphi_i,\varphi_i,\prod_{i=1}^dx_i)
=
\sum_{i=0}^{N} \alpha_i^2\,{\cal B}(\varphi_i,\varphi_i,\prod_{i=1}^dx_i).
\]
Let \(\mathcal{D}=(0,\mathbf{x}_{\max})\).
Using the definition, we obtain
\begin{align*}
\sum_{i=0}^{N} \frac{d\alpha_i}{dt}
\int_D \prod_{i=1}^dx_i\varphi_i\,dx
=
\sum_{i=0}^{N} \alpha_i^2 I_i^{\text{gain}}
-
\sum_{i=0}^{N} \alpha_i^2 I_i^{\text{loss}},
\end{align*}
where
\[
I_i^{\text{gain}}
=
\int_{\mathcal{D}} \prod_{i=1}^dx_i
\int_{\mathbf{x}}^{\mathbf{x}_{\max}}
\int_{\mathcal{D}}
{\beta}(\mathbf{x},\mathbf{y};\mathbf{z})
\Gamma(\mathbf{y},\mathbf{z})
\varphi_i(\mathbf{z})\varphi_i(\mathbf{y})\,d\mathbf{y} d\mathbf{z} d\mathbf{x},
\]
\[
I_i^{\text{loss}}
=
\int_{0}^{\mathbf{x}_{\max}} \prod_{i=1}^dx_i\varphi_i(\mathbf{x})
\int_{\mathbf{x}}^{\mathbf{x}_{\max}}\Gamma(\mathbf{x},\mathbf{y})
\varphi_i(\mathbf{y})\,d\mathbf{y}d\mathbf{x}.
\]
Changing the order of integration and using the property \eqref{eq:mass}, the gain and loss
terms cancel exactly, resulting in the following.
\[
\sum_{i=0}^{N} \frac{d\alpha_i}{dt}
\int_\mathcal{D} \prod_{i=1}^dx_i\varphi_i\,d\mathbf{x} = 0.
\]
Since \(u_h=\sum_{i=0}^{N}\alpha_i\varphi_i\), it follows that
\[
\frac{d}{dt}
\int_\mathcal{D} \prod_{i=1}^dx_iu_h\,d\mathbf{x} = 0.
\]
Hence, the total hypervolume remains constant in time. The semidiscrete scheme, therefore, preserves the fundamental hypervolume invariant of the continuous model.

\subsection*{Number Preservation}
To analyze the particle number evolution, we take \(\phi_j(\mathbf{x})=1\) in equation \eqref{semi3}.
Substituting into the weak formulation gives
\begin{align*}
\sum_{i=0}^{N} \frac{d\alpha_i}{dt}
\int_\mathcal{D} \varphi_i\,dx
=
\sum_{i=0}^{N} \alpha_i^2 J_i^{\text{gain}}
-
\sum_{i=0}^{N} \alpha_i^2 J_i^{\text{loss}},
\end{align*}
where
\[
J_i^{\text{gain}}
=
\int_{\mathcal{D}}
\int_{\mathbf{x}}^{\mathbf{x}_{\max}}
\int_{\mathcal{D}}
{\beta}(\mathbf{x},\mathbf{y};\mathbf{z})
\Gamma(\mathbf{y},\mathbf{z})
\varphi_i(\mathbf{x})\varphi_i(\mathbf{y})\,d\mathbf{y} d\mathbf{z} d\mathbf{x},
\]
\[
J_i^{\text{loss}}
=
\int_{\mathcal{D}} \varphi_i(\mathbf{x})
\int_{\mathcal{D}}\Gamma(\mathbf{x},\mathbf{y})
\varphi_i(\mathbf{y})\,d\mathbf{y}d\mathbf{x}.
\]
Reordering the integrals and using property \eqref{eq:multiplicity}, we obtain
\[
\sum_{i=0}^{N} \frac{d\alpha_i}{dt}
\int_\mathcal{D} \varphi_i\,d\mathbf{x}
=
\sum_{i=0}^{N} \alpha_i^2
\int_\mathcal{D} \int_\mathcal{D} (\nu(\mathbf{y},\mathbf{z})-1)
\Gamma(\mathbf{y},\mathbf{z})
\varphi_i(\mathbf{z})
\varphi_i(\mathbf{y})\,d\mathbf{y}d\mathbf{z}.
\]
Since \((\nu(\mathbf{y},\mathbf{z})-1)>0\), \(\Gamma(\mathbf{x},\mathbf{y})>0\), and \(\varphi_i\ge0\),
the right-hand side is nonnegative. Therefore,
\[
\frac{d}{dt}\int_D u_h\,d\mathbf{x} \ge 0.
\]
Thus, while collision increases the total number of particles,
the numerical scheme reproduces this growth exactly and introduces
no extra loss or gain. The method is therefore both hypervolume-conservative
and structurally consistent with the continuous collision dynamics.

\vspace{-0.3cm}
\subsection*{A Priori Error Estimate for the Semidiscrete Scheme}\vspace{-0.25cm}

We now derive an a priori error estimate for the semidiscrete Galerkin approximation. The result quantifies the accuracy of the finite element solution in terms of the mesh parameter $h$ and the regularity of the exact solution. We recall the following standard result regarding the $L^2$ projection. Let $\Pi_h : L^2( \mathcal{D}) \to V_h$ denote the {$L^2$ projection} defined by
\[
(\Pi_h u , \phi_h) = (u , \phi_h), \quad \forall \phi_h \in V_h.
\]

\begin{lemma} \cite{thomee2007galerkin} Let $\Pi_h$ be the standard $L^2$ projection onto a finite element space $V_h$ consisting of piecewise polynomials of degree at most $r$. If $u \in H^{s+1}(\mathcal{D})$, $1 \leq s \leq r$, then:
\[
    \|u - \Pi_h u\| \leq C h^{s+1} \|u\|_{H^{s+1}(\mathcal{D})},
\]
where $h$ is the mesh size and $C$ is a positive constant independent of $h$.\label{lem:ritz}
\end{lemma} 
\begin{theorem} 
Let $u$ be the exact solution of the continuous problem and $u_h \in V_h$ be the semidiscrete Galerkin solution. Assume $u, u_t \in L^\infty(0, T; H^{r+1}(\mathcal{D}))$, $r > 1$. Then for every $t \in [0, T]$, the following estimate holds:
\[
    \|u(t) - u_h(t)\| \leq C h^{r+1} \left( \|u\|_{H^{r+1}} + (\int_0^t (\|u\|_{H^{r+1}} + \|u_t\|_{H^{r+1}}) \, ds)^{1/2} \right),
\]
\end{theorem}

\begin{proof}
We decompose the error as:
\[
    u - u_h = (u - \Pi_h u) + (\Pi_h u - u_h) = \eta + \theta,
\]
where $\eta := u - \Pi_h u$ and $\theta := \Pi_h u - u_h$.

Subtracting the semidiscrete equation \eqref{eq:weak}   for $u_h$ from the weak formulation \eqref{eq:exactweak} and setting $\phi_h = \theta$, we obtain
\[
    (\partial_t \theta, \theta) = -[\mathcal{A}(u, u; \theta) - \mathcal{A}(u_h, u_h; \theta)] + \mathcal{B}(u, u; \theta) - \mathcal{B}(u_h, u_h; \theta) - (\partial_t \eta, \theta).
\]
Now, consider the difference term:
\begin{align*}
\mathcal{A}(u, u; \theta) - \mathcal{A}(u_h, u_h; \theta)
&= \int_0^{\mathbf{x}_{\max}} \theta\, u(\mathbf{x},t)
\left( \int_0^{\mathbf{x}_{\max}} \Gamma(\mathbf{x},\mathbf{y}) u(\mathbf{y},t)\, d\mathbf{y} \right) d\mathbf{x} \\
&\quad - \int_0^{\mathbf{x}_{\max}} \theta\, u_h(\mathbf{x},t)
\left( \int_0^{\mathbf{x}_{\max}} \Gamma(\mathbf{x},\mathbf{y}) u_h(\mathbf{y},t)\, d\mathbf{y} \right) d\mathbf{x} .
\end{align*}
Using the identity $ab - cd = (a-c)b + c(b-d)$, let:
\begin{align*}
    a &= \int_0^{\mathbf{x}_{max}} \theta u(\mathbf{x}, t), \quad b = \int_0^{\mathbf{x}_{max}} \Gamma(\mathbf{x}, \mathbf{y}) u(\mathbf{y}, t), \\
    c &= \int_0^{\mathbf{x}_{max}} \theta u_h(\mathbf{x}, t), \quad d = \int_0^{\mathbf{x}_{max}} \Gamma(\mathbf{x}, \mathbf{y}) u_h(\mathbf{y}, t).
\end{align*}
Following the decomposition of $\mathcal{A}$, we get
\begin{align*}
    \mathcal{A}(u, u; \theta) - \mathcal{A}(u_h, u_h; \theta) &= \mathcal{A}(u - u_h, u; \theta) + \mathcal{A}(u_h, u - u_h; \theta) \\
    &= \mathcal{A}(e, u; \theta) + \mathcal{A}(u_h, e; \theta),
\end{align*}
where $e = u - u_h$. By the bounds on $\mathcal{A}$,
\[
    |\mathcal{A}(e, u; \theta)| \leq C_0 (|\mathcal{D}|)^{1/2}\|e\| \|u\| \|\theta\|, \quad |\mathcal{A}(u_h, e; \theta)| \leq C_0 
    (|\mathcal{D}|)^{1/2}\|e\| \|u_h\| \|\theta\|
\]
where $|\mathcal{D}| = \prod_{i=1}^{d} x_{i,\max}$.
Combining the above estimates, we deduce
\[
|\mathcal{A}(u, u; \theta) - \mathcal{A}(u_h, u_h; \theta)|
\le C_{\cal A}\, \|e\| \, \|\theta\|,
\]
where the constant $C_A > 0$ depends on $C_0$, $b_0$, $|\mathcal{D}|$, and the bounds of $\|u\|$ and $\|u_h\|$.

Proceeding analogously, we consider the difference
\begin{align*}
\mathcal{B}(u, u; \theta) - \mathcal{B}(u_h, u_h; \theta)
&= \mathcal{B}(u - u_h, u; \theta) + \mathcal{B}(u_h, u - u_h; \theta) \\
&= \mathcal{B}(e, u; \theta) + \mathcal{B}(u_h, e; \theta),
\end{align*}
Using the boundedness properties of the operator $\mathcal{B}(\cdot,\cdot;\cdot)$, we obtain
\begin{align*}
|\mathcal{B}(e, u; \theta)| 
&\le C_0 b_0 |\mathcal{D}|^{3/2} \, \|e\| \, \|u\| \, \|\theta\|, \\
|\mathcal{B}(u_h, e; \theta)| 
&\le C_0 b_0 |\mathcal{D}|^{3/2} \, \|u_h\| \, \|e\| \, \|\theta\|,
\end{align*}
where $|\mathcal{D}| = \prod_{i=1}^{d} x_{i,\max}$ denotes the measure of the truncated domain.

\noindent
Combining the above estimates, we deduce
\[
|\mathcal{B}(u, u; \theta) - \mathcal{B}(u_h, u_h; \theta)|
\le C_B \, \|e\| \, \|\theta\|,
\]
where the constant $C_{\cal B} > 0$ depends on $C_0$, $b_0$, $|\mathcal{D}|$, and the bounds of $\|u\|$ and $\|u_h\|$.

\noindent Using above bounds, after applying the triangle inequality:
\begin{equation*}
    \frac{1}{2} \frac{d}{dt} \|\theta\|^2 \leq C_{\cal A}  \|e\| \|\theta\| + C_ {\cal B }\|e\| \|\theta\| + \|\eta_t\| \|\theta\|.
\end{equation*}
Using $e = \eta + \theta$ and $K:=C_{\cal A} +C_ {\cal B}$, we have
\begin{equation*}
    \frac{1}{2} \frac{d}{dt} \|\theta\|^2 \leq K (\|\eta + \theta\| )\|\theta\| + \|\eta_t\| \|\theta\|.
\end{equation*}
Using Young's inequality, we get
\begin{align*}
\frac{1}{2}\frac{d}{dt}\|\theta\|^2 
\le K\|\theta\|^2 + \frac{1}{2}\|\theta\|^2 + \frac{1}{2}(K\|\eta\| + \|\eta_t\|)^2.
\end{align*}
After rearranging, we have
\begin{align*}
\frac{d}{dt}\|\theta\|^2 
&\le (2K+1)\|\theta\|^2 + 2K^2\|\eta\|^2 + 2\|\eta_t\|^2.
\end{align*}
 Using Gronwall’s Lemma, it follows that 
\begin{align*}
\|\theta(t)\|^2 
\le \int_0^t e^{(2K+1)(t-s)} 
\left(2K^2\|\eta(s)\|^2 + 2\|\eta_t(s)\|^2 \right) ds.
\end{align*}
Taking square root, we obtain
\[
\|\theta(t)\| 
\le k 
\left( \int_0^t e^{(2K+1)(t-s)}\|\eta(s)\|^2 ds 
+ \int_0^te^{(2K+1)(t-s)} \|\eta_t(s)\|^2 ds \right)^{1/2},
\]
where $k>0$ is a constant depending on $K$.

Finally, using the triangle inequality $e = \eta + \theta$, we get
\[
\|e(t)\| \le \|\eta(t)\| + \|\theta(t)\|,
\]
and hence
\[
\|e(t)\| 
\le \|\eta(t)\| 
+  k 
\left( \int_0^t e^{(2K+1)(t-s)}\|\eta(s)\|^2 ds 
+ \int_0^te^{(2K+1)(t-s)} \|\eta_t(s)\|^2 ds \right)^{1/2}.
\]
Employing the standard interpolation estimates
\[
\|\eta(t)\| \le C h^{r+1}\|u(t)\|_{r+1}, 
\qquad 
\|\eta_t(t)\| \le C h^{r+1}\|u_t(t)\|_{r+1},
\]
we conclude that
\[
\|u(t) - u_h(t)\| \le C h^{r+1} \left( \|u(t)\|_{r+1} 
+ 
\left( \int_0^t \|u(s)\|_{r+1}^2 ds 
+ \int_0^t \|u_t(s)\|_{r+1}^2 ds \right)^{1/2}\right).
\]
which yields the desired convergence estimate.
\end{proof}

\section{\large\textbf{Fully Discrete Error Analysis}}\label{sec:4}
\vspace{-0.4cm}

We now derive the fully discrete scheme corresponding to the semidiscrete Galerkin approximation of the multidimensional nonlinear collisional fragmentation equation. 
Spatial discretization is performed using finite element space $V_h$, while temporal integration is implemented using the second-order backward differentiation formula (BDF2). Let $T>0$ denote the final time and $N \in \mathbb{N}$ the number of time steps. We introduce a uniform partition of the time interval $[0,T]$ given by
\[
\tau = \frac{T}{N}, 
\qquad 
t^n = n\tau, 
\qquad n = 0,1,\dots,N.
\]
For a sufficiently smooth function $w:[0,T]\to L^2(\mathcal{D})$, we define the discrete values and time-stepping operators as
\[
w^n := w(\cdot,t^n), 
\qquad 
\delta_t w^n := \frac{w^n - w^{n-1}}{\tau},
\]
and
\[
D_t^{(2)} w^n 
:= 
\frac{3w^n - 4w^{n-1} + w^{n-2}}{2\tau}, 
\qquad n \ge 2.
\]
Here, $w^n$ denotes the approximation at time level $t^n$, while $D_t^{(2)} w^n$ provides a second-order accurate approximation of the time derivative $\partial_t w(t^n)$.

\subsection*{BDF2 Fully Discrete Scheme}

The fully discrete finite element approximation seeks
$U_h^n \in V_h$ such that for all $\phi_h\in V_h$,
\begin{equation}
\label{FD-BDF2}
\begin{cases}
\big(D_t^{(2)} U_h^n,\phi_h\big)
+
{\cal A}(U_h^n,U_h^n,\phi_h)
=
{\cal B}(U_h^n,U_h^n,\phi_h),
& n\ge 2, \\[8pt]

\big(\delta_t U_h^1,\phi_h\big)
+
{\cal A}(U_h^1,U_h^1,\phi_h)
=
{\cal B}(U_h^1,U_h^1,\phi_h),
& n=1,\\[8pt]
U_h^0 = \Pi_h u_{\mathrm{in}},
\end{cases}
\end{equation}
The initial time step is computed using the backward Euler method
to initialize the BDF2 scheme, which is ensuring second-order temporal accuracy for $n\ge 2$. In the later analysis, we demonstrate its stability and derive error estimates under suitable regularity assumptions on the exact solution.

\vspace{-0.3cm}
\paragraph{\textbf{Algebraic Structure of the Fully Discrete Scheme:}} Substituting the finite element expansion 
$
u_h^n(\mathbf{x})=\sum_{i=1}^{M}\alpha_i^n\,\varphi_i(\mathbf{x}),
$ into the fully discrete formulation and choosing $\phi_h=\varphi_j$, $j=1,\dots,M$, we obtain a system of nonlinear algebraic equations
\[
 \sum_{i=1}^{M} \frac{1}{\tau}(\varphi_i,\varphi_j)\,\alpha_i^n
+
\sum_{i=1}^{M} \mathcal{A}(\varphi_i,\varphi_i,\varphi_j)\,(\alpha_i^n)^2
-
\sum_{i=1}^{M} \mathcal{B}(\varphi_i,\varphi_i,\varphi_j)\,(\alpha_i^n)^2
= \sum_{i=1}^{M} \frac{1}{\tau}(\varphi_i,\varphi_j)\,\alpha_i^{n-1}
\; j=1,\dots,M.\quad 
\]
\textbf{First time step ($n=1$).}
The scheme reduces to a backward Euler-type formulation:
\[
\frac{1}{\tau}\mathbf{M}\boldsymbol{\alpha}^n 
+ (\mathbf{A}-\mathbf{B})\,(\boldsymbol{\alpha}^n)^{\odot 2}
= \frac{1}{\tau}\mathbf{M}\boldsymbol{\alpha}^{n-1}.
\]
Equivalently, we define the nonlinear residual
\[
\mathbf{F}(\boldsymbol{\alpha}^n)
:=
\mathbf{M}\frac{\boldsymbol{\alpha}^n-\boldsymbol{\alpha}^{n-1}}{\tau}
+ (\mathbf{A}-\mathbf{B})\,(\boldsymbol{\alpha}^n)^{\odot 2}.
\]

\medskip
\noindent
\textbf{Subsequent time steps ($n\ge 2$).}
Employing the BDF2 scheme, we obtain
\[
\mathbf{M}\frac{3\boldsymbol{\alpha}^n - 4\boldsymbol{\alpha}^{n-1} + \boldsymbol{\alpha}^{n-2}}{2\tau}
+ (\mathbf{A}-\mathbf{B})\,(\boldsymbol{\alpha}^n)^{\odot 2}
= \mathbf{0}.
\]
This can be compactly written as the nonlinear system
\[
\mathbf{G}(\boldsymbol{\alpha}^n)
:=
\mathbf{M}\frac{3\boldsymbol{\alpha}^n - 4\boldsymbol{\alpha}^{n-1} + \boldsymbol{\alpha}^{n-2}}{2\tau}
+ (\mathbf{A}-\mathbf{B})\,(\boldsymbol{\alpha}^n)^{\odot 2}.
\]
\paragraph{\textbf{Existence and Uniqueness:}}
At each time level, the fully discrete problem reduces to solving a finite-dimensional nonlinear system. We employ the Newton--Raphson method: given an initial guess $\boldsymbol{\alpha}^{(m)}$, the update is defined by
\[
\boldsymbol{\alpha}^{(m+1)} 
= 
\boldsymbol{\alpha}^{(m)} 
- \mathbf{J}^{-1}(\boldsymbol{\alpha}^{(m)})\,\mathbf{G}(\boldsymbol{\alpha}^{(m)}),
\]
where $\mathbf{J}$ denotes the Jacobian of $\mathbf{G}$. Since the mass matrix $\mathbf{M}$ is symmetric positive definite, it is invertible. Moreover, the nonlinear operators are bounded, and the leading linear term 
$
\frac{3}{2\tau}\mathbf{M}
$
dominates the system for sufficiently small $\tau$. Consequently, the Jacobian matrix $\mathbf{J}$ is coercive, i.e.,
$
\mathbf{v}^T \mathbf{J} \mathbf{v} > 0 \quad \forall\, \mathbf{v} \neq 0,
$
which guarantees the existence and uniqueness of the discrete solution at each time step and the local convergence of the Newton iteration.
\begin{lemma}[Stability of the fully discrete BDF2 scheme]
Let $\{U_h^n\}_{n\ge0}\subset V_h$ be the solution of the fully discrete
scheme \eqref{FD-BDF2}. Assume that $\Delta t < \frac{1}{4K}.$ Then the discrete solution satisfies
\[
\|U_h^{1}\|
\le (1-K\Delta t)^{-1/2}\|U_h^{0}\|,
\quad
|||U_h^n|||^2
\le\frac{|||U_h^{1}|||^2}{1-4K\Delta t}
\exp\!\left(\frac{4K\Delta t(N-2)}{1-4K\Delta t}\right)
,
\quad t_n=n\Delta t,
\]
where $
|||U_h^n|||^2 := \|U_h^n\|^2 + \|2U_h^n-U_h^{\,n-1}\|^2.
$
Consequently, for any $0\le t_n\le T$,
\[
\|U_h^n\|\le K_T\|U_h^0\|,
\]
where $K_T>0$ is independent of $h$ . Hence, the fully discrete  BDF2 scheme is stable on any finite time interval.
\end{lemma}
\begin{proof} For $n\ge2$ testing \eqref{FD-BDF2} with $\phi_h=U_h^n$ yields
\begin{equation}\label{eq1}
\frac{1}{2\Delta t}\left( 3U_h^{n}-4U_h^{n-1}+U_h^{n-2},U_h^{n}\right)
+ A(U_h^{n},U_h^{n},U_h^{n})
= B(U_h^{n},U_h^{n},U_h^{n}).
\end{equation}
First, we observe that
\[
\left(
\frac{3U_h^n - 4U_h^{n-1} + U_h^{n-2}}{2\Delta t},
U_h^n
\right)
=
\frac{1}{2\Delta t}
\left(
3\|U_h^n\|^2
-4(U_h^{n-1},U_h^n)
+(U_h^{n-2},U_h^n)
\right).
\]
We can rewrite as 
\[\left(
\frac{3U_h^n - 4U_h^{n-1} + U_h^{n-2}}{2\Delta t},
U_h^n
\right)
=
\frac{1}{4\Delta t}
\left(
6\|U_h^n\|^2
-8(U_h^{n-1},U_h^n)
+2(U_h^{n-2},U_h^n)
\right).
\]
Using the identity
$
(a-2b+c)^2
=
a^2 + 4b^2 + c^2
-4ab -4bc +2ac,
$ with $a=U_h^n$, $b=U_h^{n-1}$, $c=U_h^{n-2}$, 
we obtain
\[
\left(
\frac{3U_h^n - 4U_h^{n-1} + U_h^{n-2}}{2\Delta t},
U_h^n
\right)
=
\frac{1}{4\Delta t}
\Big[
|||U_h^n|||^2
- |||U_h^{n-1}|||^2
+ \|U_h^n -2U_h^{n-1}+U_h^{n-2}\|^2
\Big].
\]
Substituting into \eqref{eq1} yields
\[
\frac{1}{4\Delta t}
\Big[
|||U_h^n|||^2
- |||U_h^{n-1}|||^2
+ \|U_h^n -2U_h^{n-1}+U_h^{n-2}\|^2
\Big]
+ {\cal A}(U_h^n,U_h^n,U_h^n)
=
{\cal B}(U_h^n,U_h^n,U_h^n).
\]
Here,
\[
 |||U_h^{n}|||\!^{2}:=\|U_h^{n}\|^{2}+\|2U_h^{n}-U_h^{n-1}\|^{2}
.
\]
If $|\mathcal{D}| = \prod_{i=1}^{d} x_{i,\max}$ denotes the measure of the truncated domain.

\noindent Using bounds of $\mathcal{A}$ and $\mathcal{B}$ , and $C=C_0 b_0 (\mathcal{D})^{3/2}+C_0(\mathcal{D})^{1/2}$,we get
\[
|||U_h^{n}|||^{2}-|||U_h^{n-1}|||^{2}
\le 4\Delta t C\|U_h^{n}\|^{3}.
\]
As $\mathcal{D}$ is bounded and $U_h$ is continuous,
we can assume $||U_h^{n}||^3\le||U_h^{n}||\quad|||U_h^{n}|||^2\le M|||U_h^{n}|||^2$,
\[
|||U_h^{n}|||^{2}-|||U_h^{n-1}|||^{2}
\le 4\Delta t K|||U_h^{n}|||^2.
\]
Taking summation from \(n=2\) to \(N\),
\[
|||U_h^{N}|||^{2}-|||U_h^{1}|||^{2}
\le 4\Delta tK
\sum_{n=2}^{N}\|U_h^{n}\|^{2}.
\]
We deduce that
\[
|||U_h^{N}|||^{2}
\le \frac{|||U_h^{1}|||^{2}}{1-4\Delta tK}+\frac{4\Delta tK}{1-4\Delta tK}
\sum_{n=2}^{N-1}\|U_h^{n}\|^{2}.
\]
Using Discrete Grownwall lemma and $(1+x)\le exp(x) \quad \forall x\ge0$
\begin{equation}
|||U_h^{N}|||^{2}
\le \frac{|||U_h^{1}|||^{2}}{1-4\Delta tK}exp(\frac{4\Delta tK}{1-4\Delta tK}(N-2)).
\label{eq:stability}
\end{equation}
For n=1,
\[
\|U_h^1\|^2 - \|U_h^0\|^2 \le 2C\Delta t \|U_h^1\|^3,
\]
we get,
\[
\|U_h^1\|^2 \le (1-2K\Delta t)^{-1}\|U_h^0\|^2 
,\]
using triangle inequality,
\[
|||U_h^1|||^2 \le \|U_h^{1}\|^2 +
\left(2\|U_h^{1}\| +\|U_h^{0}\| \right)^2
\le
\frac{\|U_h^0\|^2}{1-2K\Delta t}
+
\left(
\frac{2}{\sqrt{1-2K\Delta t}} + 1
\right)^2
\|U_h^0\|^2
,\]
\[
|||U_h^1|||^2 
\le
\|U_h^0\|^2
\left(
\frac{1}{1-2K\Delta t}
+
\left(
\frac{2}{\sqrt{1-2K\Delta t}} +1
\right)^2
\right).
\]
Using this on \ref{eq:stability}, we get
 \[
\|U_h^n\|\le K_T\|U_h^0\|,
\]
\end{proof}
\vspace{-0.3cm}

\subsection{Error Analysis of Fully-Discrete Formulation}
\vspace{-0.3cm}

We now derive an a priori error estimate for the fully discrete scheme. We decompose the total error into two distinct parts: the spatial projection error and the temporal discretization error.
The error at time $t_n$ is then decomposed as:
$e= U_h^n - u^n:= \theta^n + \eta^n,
$ where
$\theta^n := U_h^n - \Pi_h u^n, \;\; 
\eta^n := \Pi_h u^n - u^n,$
\begin{theorem}[\textbf{Fully Discrete Convergence for BDF2 Scheme}]\label{thm:bdf2_convergence}
Let $u$ be the solution of the nonlinear parabolic problem \eqref{eq:model}, and let $U_h^n \in V_h$ denote the fully discrete finite element solution obtained using the BDF2 scheme \eqref{FD-BDF2}. Assume
\[
u \in H^1(0,T;H^{s+1}( \mathcal{D})) \cap H^3(0,T;L^2( \mathcal{D})), \quad s \in \{1,2,\ldots,r\}.
\]
Then, for sufficiently small $\Delta t$, the fully discrete error 
$ \mathbf{E}_h^n := U_h^n - u^n$
satisfies
\begin{eqnarray}
\|U_h^n - u^n\|^2 
&\le& C h^{2(r+1)}\left(\|u_0\|_{H^{r+1}}^2
+\int_{0}^{t_N}\|u_t(s)\|_{H^{r+1}}^2\,ds\,\right) + O(\Delta t^4) .
\end{eqnarray}
\end{theorem}
\begin{proof}
Let's recall the error splitting
$
\theta^n := U_h^n - {\Pi}_h u^n, \quad \eta^n := {\Pi}_h u^n - u^n.
$

Using Lemma \ref{lem:ritz}, we get
\begin{eqnarray} \label{tau3}
\|\eta^n\| \le Ch^{r+1}\|u(t_n)\|_{H^{r+1}(\mathcal{D})}
\le Ch^{r+1} \left(\|u_0\|_{H^{r+1}(\mathcal{D} )}+\int_0^{t_n}\|u_t(s)\|_{H^{r+1}(\mathcal{D} )}ds\right),
\end{eqnarray}
consequently
\begin{eqnarray} \label{eta}
\Delta t\sum_{n=2}^{N}\|\eta_t^n\|^2
\le
C h^{2(r+1)}\sum_{n=2}^{N}\int_{t_{n-2}}^{t_n}\|u_t(s)\|_{H^{r+1}}^2\,ds
\le
C h^{2(r+1)}\int_{0}^{t_N}\|u_t(s)\|_{H^{r+1}}^2\,ds .
\end{eqnarray}
For $n\ge2$ testing \eqref{FD-BDF2} with $\phi_h=\theta^n$ yields
\[
(\partial_t U_h^n,\theta^n) + \mathcal{A}(U_h^n,U_h^n,\theta^n)
= \mathcal{B}(U_h^n,U_h^n,\theta^n).
\]
Similarly, for the exact solution
\[
(\partial_t u^n,\theta^n) + \mathcal{A}(u^n,u^n,\theta^n)
= \mathcal{B}(u^n,u^n,\theta^n).
\]
Subtracting the two equations,
\[
(\partial_t(U_h^n-u^n),\theta^n)
+ \mathcal{A}(U_h^n,U_h^n,\theta^n)-\mathcal{A}(u^n,u^n,\theta^n)
= \mathcal{B}(U_h^n,U_h^n,\theta^n)-\mathcal{B}(u^n,u^n,\theta^n).
\]
We obtain
\[
(\partial_t\theta^n,\theta^n) + (\partial_t\eta^n,\theta^n)
+ \mathcal{A}(U_h^n,U_h^n,\theta^n)
- \mathcal{A}(u^n,u^n,\theta^n)
= \mathcal{B}(U_h^n,U_h^n,\theta^n)
- \mathcal{B}(u^n,u^n,\theta^n).
\]
Using BDF2 discretization,
\[
\frac{1}{2\Delta t}
(3\theta^n-4\theta^{n-1}+\theta^{n-2},\theta^n)
+ \mathcal{A}(U_h^n,U_h^n,\theta^n)
- \mathcal{A}(u^n,u^n,\theta^n)
\nonumber\\
= \mathcal{B}(U_h^n,U_h^n,\theta^n)
- \mathcal{B}(u^n,u^n,\theta^n)
- (\partial_t\eta^n,\theta^n).
\]
Using the identity,
\begin{align*}
\frac{1}{2\Delta t}(3\theta^n-4\theta^{n-1}+\theta^{n-2},\theta^n)
&=
\frac{1}{4\Delta t}
\Big[
\|\theta^n\|^2
+\|2\theta^n-\theta^{n-1}\|^2
-\|\theta^{n-1}\|^2 \nonumber\\
&\qquad
-\|2\theta^{n-1}-\theta^{n-2}\|^2
\Big]
+\frac{1}{4\Delta t}\|\theta^n-2\theta^{n-1}+\theta^{n-2}\|^2 .
\end{align*}
Therefore
\begin{align*}
\frac{1}{4\Delta t}
\Big[
(\|\theta^{n}\|^{2}
+
\|2\theta^{n}-\theta^{n-1}\|^{2})
&-
(\|\theta^{n-1}\|^{2}
+
\|2\theta^{n-1}-\theta^{n-2}\|^{2})
\Big]  \le
\mathcal{B}(e,U^{n},\theta^{n})
+
\mathcal{B}(U_h^{n},e,\theta^{n}) \\
&\quad
-
\mathcal{A}(U_h^{n},e,\theta^{n})
-
\mathcal{A}(e,U^{n},\theta^{n})
+
\frac{1}{2\Delta t}(3\eta^{n}-4\eta^{n-1}+\eta^{n-2},\theta^{n}).
\end{align*}
Let $
||| \theta^n |||^2
:=
\|\theta^n\|^2 + \|2\theta^n-\theta^{n-1}\|^2,
$
and
\[
\frac{3\eta^n-4\eta^{n-1}+\eta^{n-2}}{2\Delta t}
=
\frac{\partial \eta}{\partial t}
+ TE
=
\frac{\partial \eta}{\partial t}
+ O(\Delta t^2).
\]
By using bounds of $\mathcal{A}(\cdot, \cdot, \cdot)$ and $\mathcal{B}(\cdot, \cdot, \cdot)$, we obtain
\[
\frac{1}{4\Delta t}\left(|||\theta^{n}|||^2-|||\theta^{n-1}|||^2\right)
\le
K_1\|\theta^n\|^2
+
K_2\|\eta^n\|^2
+ K_3\|\frac{\partial \eta}{\partial t}\|^2+ O(\Delta t^4)
.
\]
Summing over n = 1, . . . , N and noting that the first term forms a telescoping sum, we obtain
\begin{eqnarray}
||| \theta^N |||^2 - ||| \theta^0 |||^2
\le C(\nu) \Delta t \sum_{n=1}^{N}|||\theta^n|||^2+
 C_\nu h^{2(r+1)}\Big(\int_{0}^{t_N}\|u_t(s)\|_{H^{r+1}(\mathcal{D})}^2\,ds \nonumber\\
 +\Delta t\sum_{n=1}^{N}(\|u_0\|_{H^{r+1}(\mathcal{D})} + \int_0^{t_n}\|u_t(s)\|_{H^{r+1}(\mathcal{D})}\,ds)^2\Big)
+ \sum_{n=1}^{N}O(\Delta t^5) .
\end{eqnarray}

Since $\theta^0=0$ and $\|\theta^n\|^2 \le ||| \theta^n |||^2$, the above inequality simplifies to
\begin{eqnarray}
||| \theta^N |||^2 
\le C(\nu) \Delta t \sum_{n=1}^{N}|||\theta^n|||^2+
 C_\nu h^{2(r+1)}\Big(\int_{0}^{t_N}\|u_t(s)\|_{H^{r+1}(\mathcal{D})}^2\,ds \nonumber\\
 +\Delta t\sum_{n=1}^{N}(\|u_0\|_{H^{r+1}(\mathcal{D})} + \int_0^{t_n}\|u_t(s)\|_{H^{r+1}(\mathcal{D})}\,ds)^2\Big)
+\sum_{n=1}^{N}O(\Delta t^5) .
\end{eqnarray}
Finally, applying the discrete Grönwall lemma, we conclude
\begin{eqnarray} \label{t2}
||| \theta^n |||^2 \le C h^{{2(r+1)}}(\int_{0}^{t_N}\|u_t(s)\|_{H^{r+1}(\mathcal{D})}^2\,ds+\|u_0\|_{H^{r+1}(\mathcal{D})}^2)+O(\Delta t^4).
\end{eqnarray}
By combining the estimates \eqref{tau3} and \eqref{t2}, we arrive at the desired fully discrete estimate
\begin{eqnarray}
\|U_h^n - u^n\|^2 
&\le& \|\theta^n\|^2+\|\eta^n\|^2 \nonumber\\
&\le& C h^{2(r+1)}\left(\|u_0\|_{H^{r+1}}^2
+\int_{0}^{t_N}\|u_t(s)\|_{H^{r+1}}^2\,ds\,\right) + O(\Delta t^4) .
\end{eqnarray}
which completes the proof.





\end{proof}
\vspace{-0.35cm}

\section{Numerical Results}\label{sec:5}
\vspace{-0.35cm}

In this section, we assess the numerical performance of the proposed finite element scheme in a computational domain $\mathcal{D} \subset \mathbb{R}^d$, $d \geq 1$. The primary objective is to validate the convergence properties and the spatial accuracy of the discretization. We compute the moments associated with the considered nonlinear collision integro–partial differential equation. 
Let $u$ denote the exact analytical solution and $u_h$ the corresponding finite element approximation. The discretization error is defined as $\mathcal{E}:= u - u_h.$ The error is measured in the standard norms $L^1(\mathcal{D})$, $L^2(\mathcal{D})$,  $L^\infty(\mathcal{D})$  and $H^1(\mathcal{D})$.\\ 
The $L^p$-error is defined as:
\[
\|\mathcal{E}\|_{L^p(\mathcal{D})}
=
\left(
\int_{\mathcal{D} } |\mathcal{E}|^p\ \, d\mathbf{x}
\right)^{1/p},
\quad
\text{where } p=1,2,\quad \mathbf{x}=(x_1,x_2,\dots,x_d)\in\mathbb{R}^d.
\]
The $H^1$-error is given by:
\[
\|\mathcal{E}\|_{H^1(\mathcal{D} )}
=
\left(
\|\mathcal{E}\|_{L^2(\mathcal{D})}^2
+
\sum_{i=1}^{d}
\left\|
\frac{\partial \mathcal{E}}{\partial x_i}
\right\|_{L^2(\mathcal{D})}^2
\right)^{1/2}.\]
The $L_\infty$-error is given by:
\[
\|\mathcal{E}\|_{L^\infty} = \max_{1 \leq j \leq N} |u(\mathbf{x}_j) - u_h(\mathbf{x}_j)|.\]
To obtain a normalized measure of the approximation quality, we also compute the relative $L^\infty$-error defined by
\begin{equation}
\mathrm{RelError}_{L^\infty}
=\frac{\|u-u_h\|_{L^\infty(\mathcal{D})}}
     {\|u\|_{L^\infty(\mathcal{D})}}.
\end{equation}
To quantify the asymptotic behavior of the scheme, the Experimental Order of Convergence (EOC) is computed at the final simulation time $T$. When the exact solution $u$ is available, the EOC is evaluated by comparing the errors obtained on two successive mesh refinements characterized by mesh sizes $h_N$ and $h_{2N}$:
\begin{equation}
\mathrm{EOC}
=
\frac{\ln \left( \mathcal{E}_N / \mathcal{E}_{2N} \right)}
     {\ln \left( h_N / h_{2N} \right)}.
\end{equation}
\subsection{Moments and their relative error}
Let $\mathcal{M}(t)$ denote an exact analytical moment of the non linear collision fragmentation equation and $\mathcal{M}_h(t)$ its numerical approximation computed from $u_h$. The relative error of the moment is defined as
\begin{equation}
\mathrm{RelError}_{\mathcal{M}}(t)
=
\frac{|\mathcal{M}(t)-\mathcal{M}_h(t)|}
     {|\mathcal{M}(t)|}.
\end{equation}
These quantities enable the validation of both local accuracy, through norm-based errors, and global structural accuracy, through the preservation of moments. We have examined a set of representative test cases to validate the numerical scheme, in which both analytical and numerical moments of the NCBE are computed. Comparing the exact and approximated moments provides a thorough assessment of the accuracy and consistency of the proposed finite element method.
\subsubsection{One Dimensional Cases}

\textbf{Test Case 1: Product kernel with an exponential initial distribution}

In this test case, we are considering the product collisional kernel $\Gamma(x,y)=xy$ and binary breakage kernel $\beta(x,y;z)=\frac{2}{y}$ which is a fundamental model for describing the coagulation kinetics of particals in systems such as rennet-induced coagulation behaviour across various particle sizes. Within the context of aerosol science, the product kernel effectively characterizes the physical tendency of larger particles to collide more frequently, which subsequently promote fragmentation into smaller particles. For this example, initial condition is taken as $u(x,0)=e^{-x}$ with moments as $\mathcal{M}_0(t)=1+t$ and $\mathcal{M}_1(t)=1$. The results for Case 1 are illustrated in Figure~\ref{fig:case1}. Mass is conserved throughout the simulation. The accuracy of the zeroth moment improves as the grid is refined. The relative errors of the moments computed using the BDF2 scheme are listed in Table~\ref{tab:1m0} - Table~\ref{tab:1m1} for different grid resolutions. It can be observed that the relative errors in all moments decrease as we refine the grids.
\begin{table}[!h]
\centering
\caption{Comparison of exact and numerical values of $\mathcal{M}_{0}(t)=1+t$.}
\label{tab:1m0}
\renewcommand{\arraystretch}{1.1}
\setlength{\tabcolsep}{5pt}

\begin{tabular}{c c cc cc cc}
\toprule
$t$ & Exact 
& \multicolumn{2}{c}{$80\quad grids$}
& \multicolumn{2}{c}{$160\quad grids$}
& \multicolumn{2}{c}{$320\quad grids$} \\

\cmidrule(lr){3-4} \cmidrule(lr){5-6} \cmidrule(lr){7-8}

& 
& Num & Rel. Err.
& Num & Rel. Err.
& Num & Rel. Err. \\
\midrule
 
2.0 & 3.0
& 3.0088  & 2.9199e-03
& 3.0022 & 7.2990e-04  
& 3.0005 & 1.8225e-04 \\

4.0 & 5.0
& 5.0405 & 8.1078e-03
& 5.0101 & 2.0288e-03 
& 5.0025 & 5.0729e-04\\

6.0 & 7.0
& 7.1111 & 1.5875e-02
& 7.0278  & 3.9769e-03 
& 7.0070  & 9.9468e-04 \\

8.0 & 9.0
&  9.2358 & 2.6198e-02 
& 9.0592 &  6.5730e-03
& 9.0148 &  1.6446e-03\\

10.0 & 11.0
& 11.4295  & 3.9046e-02
& 11.1080  & 9.8152e-03 
& 11.0270  & 2.4571e-03 \\

\bottomrule
\end{tabular}
\end{table}
\begin{table}[!h]
\centering
\caption{Comparison of exact and numerical values of $\mathcal{M}_{1}(t)=1$.}
\label{tab:1m1}
\renewcommand{\arraystretch}{1.1}
\setlength{\tabcolsep}{5pt}

\begin{tabular}{c c cc cc cc}
\toprule
$t$ & Exact 
& \multicolumn{2}{c}{$80\quad grids$}
& \multicolumn{2}{c}{$160\quad grids$}
& \multicolumn{2}{c}{$320\quad grids$} \\

\cmidrule(lr){3-4} \cmidrule(lr){5-6} \cmidrule(lr){7-8}

& 
& Num & Rel. Err.
& Num & Rel. Err.
& Num & Rel. Err. \\
\midrule

2.0 & 1
& 1.0028  & 2.7643e-03
& 1.0007 & 6.8521e-04
& 1.0002 & 1.6742e-04 \\

4.0 & 1
& 1.0076  & 7.5516e-03
& 1.0019  & 1.8712e-03 
& 1.0005 & 4.6614e-04\\

6.0 & 1
&  1.0148 & 1.4771e-02
&  1.0036 & 3.6345e-03
&  1.0009  & 9.0289e-04\\

8.0 & 1
& 1.0245  &  2.4541e-02
& 1.0060  & 5.9909e-03 
& 1.0015  & 1.4837e-03\\

10.0 & 1
& 1.0370  & 3.6996e-02
& 1.0090  & 8.9525e-03 
& 1.0022   & 2.2099e-03\\

\bottomrule
\end{tabular}
\end{table}
\begin{figure}[!h]
\centering
\begin{subfigure}{0.32\textwidth}
    \centering
    \includegraphics[width=\linewidth]{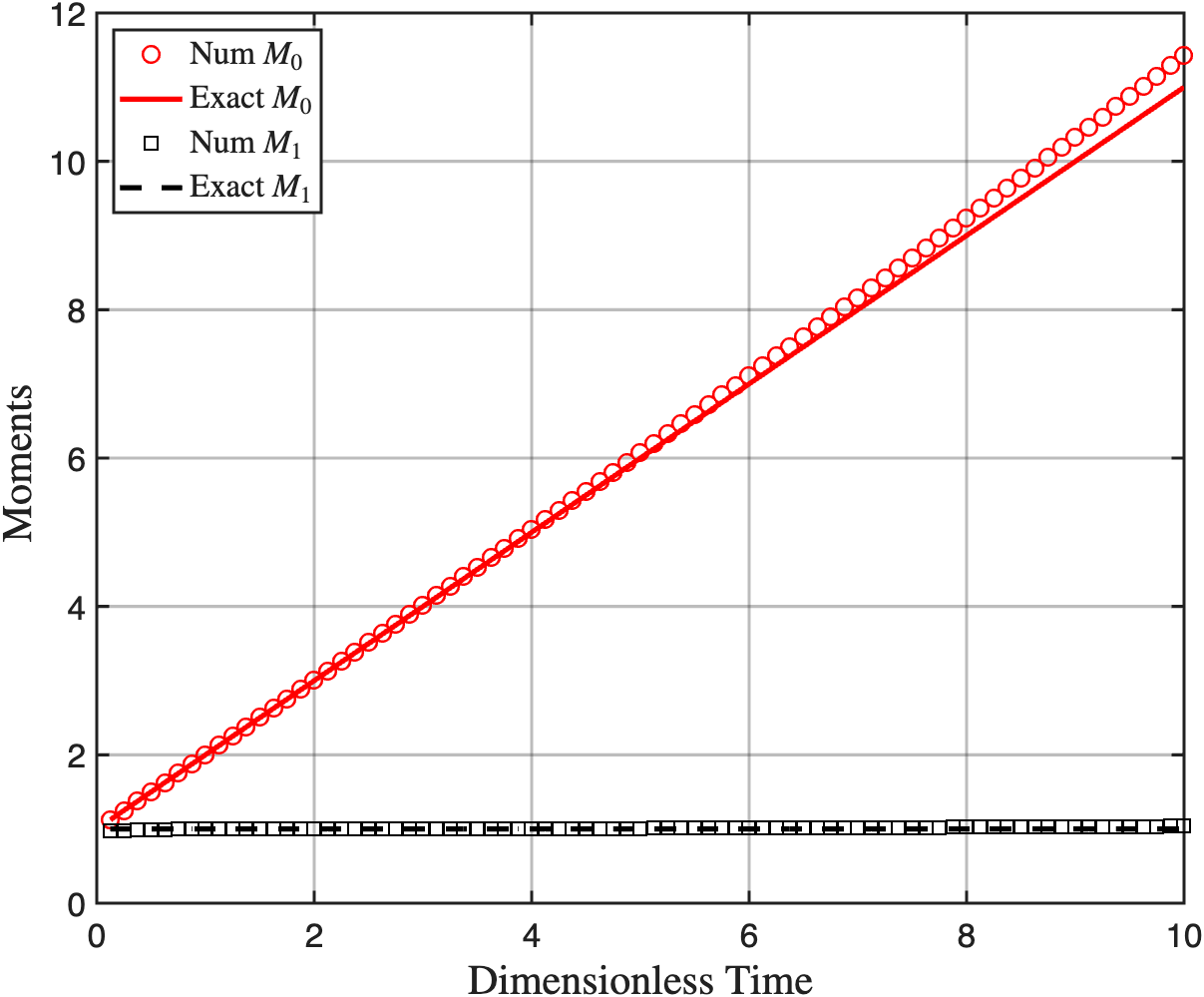}
    \caption{$80$ grids}
\end{subfigure}
\hfill
\begin{subfigure}{0.32\textwidth}
    \centering
    \includegraphics[width=\linewidth]{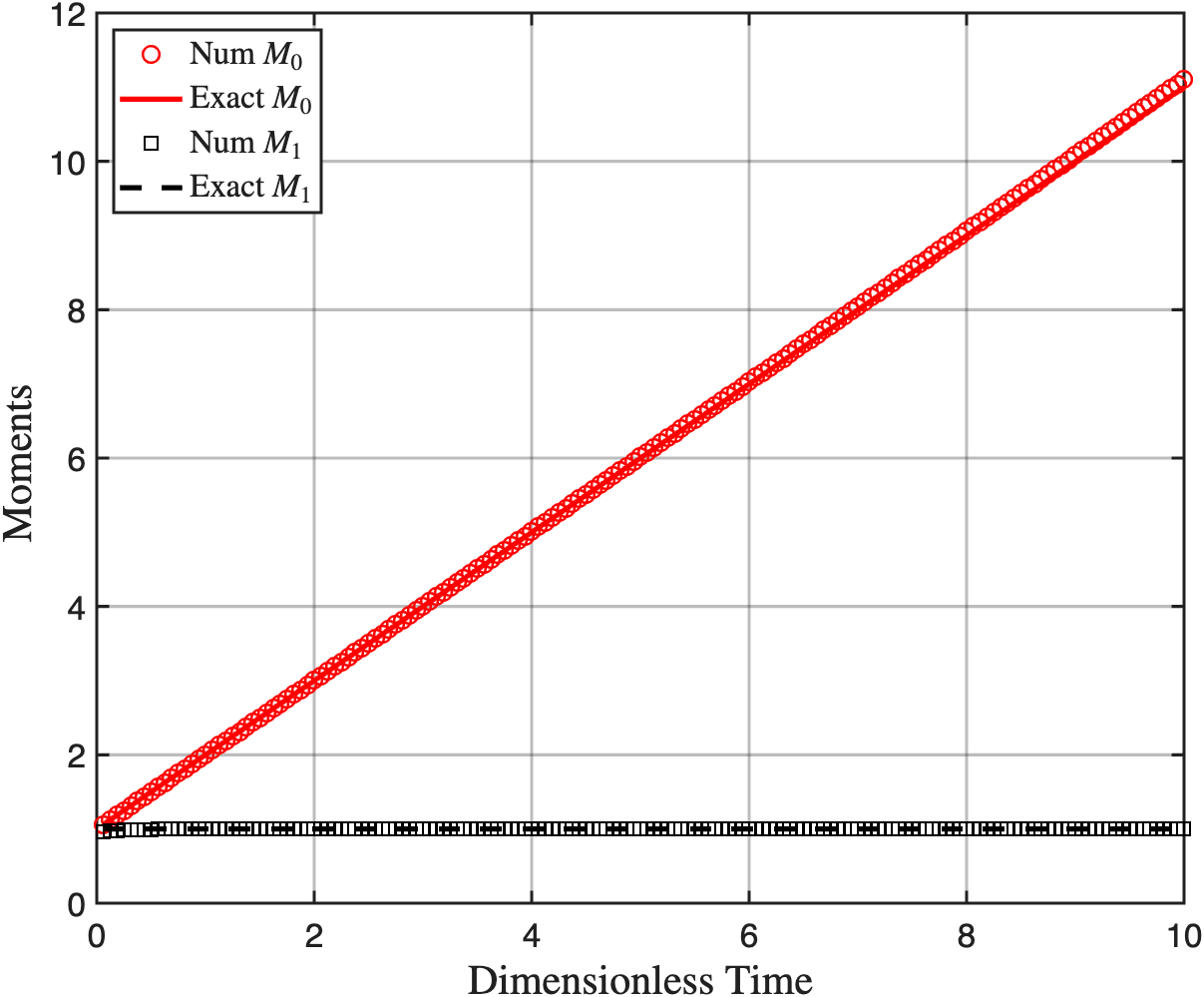}
    \caption{$160$ grids}
\end{subfigure}
\hfill
\begin{subfigure}{0.32\textwidth}
    \centering
    \includegraphics[width=\linewidth]{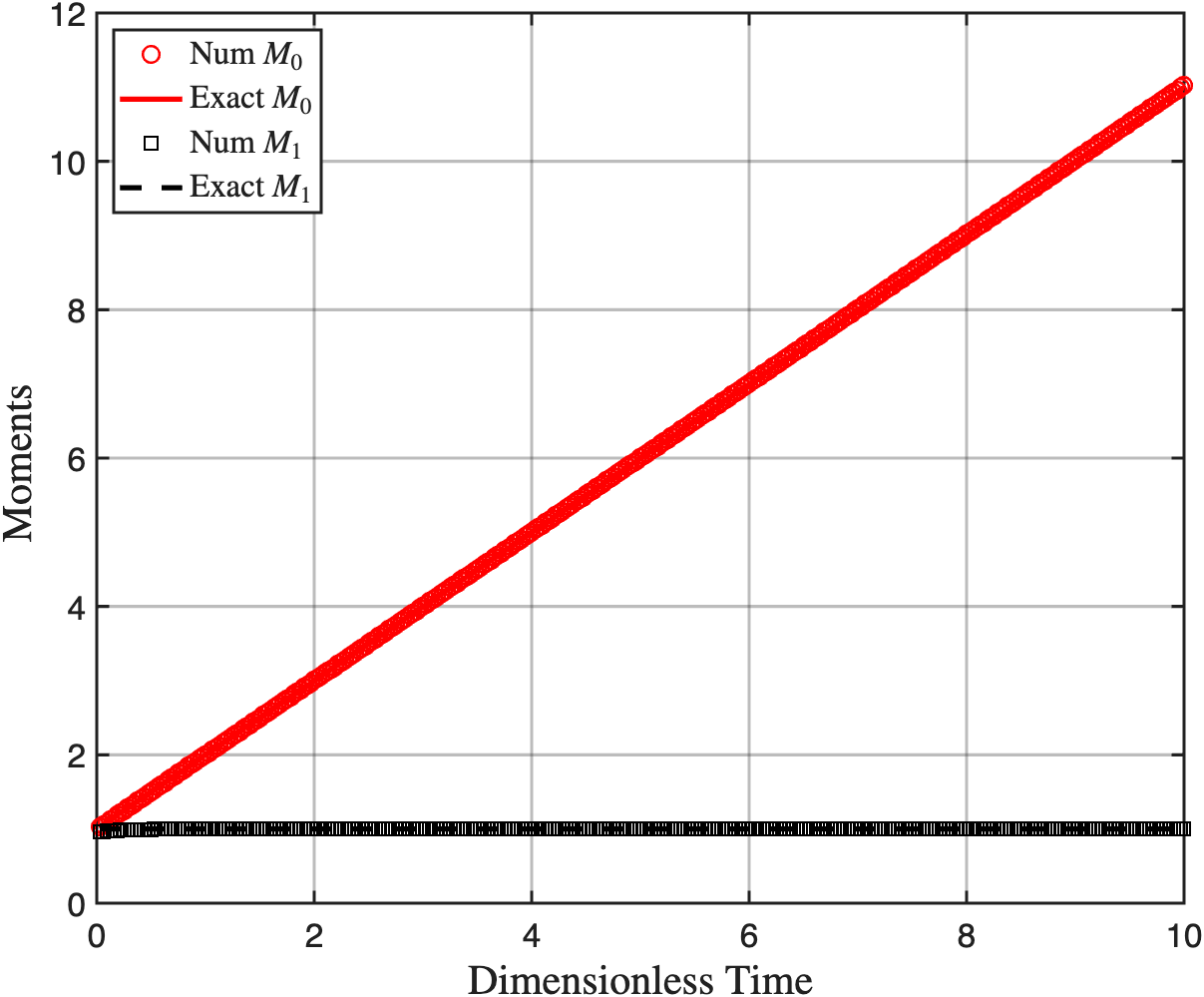}
    \caption{$320$ grids}
\end{subfigure}

\caption{Comparison of Test Case 1 for different grid resolutions.}
\label{fig:case1}
\end{figure}

\textbf{Test Case 2: Binary deterministic breakage kernel with an exponential initial distribution}

In this example, we consider the constant collision kernel $\Gamma(x,y)=1$ and dirac type breakage kernel $\beta(x,y;z)=\delta(x-0.4y)+\delta(x-0.6y)$. In literature, the exact solution for this case is not available. However, Initial condition is taken as $u(x,0)=e^{-x}$ with moments as $\mathcal{M}_0(t)=\frac{1}{1-x}$ and $\mathcal{M}_1(t)=1$. Mass is strictly conserved throughout the simulation. The accuracy of the zeroth moment improves significantly as the grid resolution is increased. The relative errors for the moments, computed using the proposed scheme, are listed in Table~\ref{tab:2m0} -Table~\ref{tab:2m1} for various grid resolutions. These results show that the relative errors across all moments decrease as the grid is refined.  The results for Case 2 are illustrated in Figure \ref{fig:case2}, which depicts the temporal evolution of the moments. Comparing the zeroth moments from FEM with those obtained from FVM, HAM, and AHPM presented in \cite{bariwal2024non} underscores the superior accuracy of the proposed scheme. While the zeroth moment curve obtained via the proposed FEM aligns perfectly with the exact solution, plots from alternative methods reveal noticeable deviations from the exact curve.

\begin{table}[htbp]
\centering
\caption{Comparison of exact and numerical values of $\mathcal{M}_{0}(t)=\frac{1}{1-t}$.}
\label{tab:2m0}
\renewcommand{\arraystretch}{1.1}
\setlength{\tabcolsep}{5pt}
\begin{tabular}{c c cc cc cc}
\toprule
$t$ & Exact 
& \multicolumn{2}{c}{$80\quad grids$}
& \multicolumn{2}{c}{$160\quad grids$}
& \multicolumn{2}{c}{$320\quad grids$} \\
\cmidrule(lr){3-4} \cmidrule(lr){5-6} \cmidrule(lr){7-8}
& 
& Num & Rel. Err.
& Num & Rel. Err.
& Num & Rel. Err. \\
\midrule
0.15 & 1.1765
& 1.1685 &  6.7832e-03  
& 1.1675   & 7.6553e-03  
& 1.2398 & 8.1714e-03  \\
0.3 & 1.4286
& 1.4214   & 5.0260e-03
& 1.4165  &  8.4154e-03
& 1.4152  &  9.3373e-03  \\
0.45 & 1.8182
&  1.8306 &  6.8034e-03 
& 1.8060 &  6.7244e-03
& 1.7986    & 1.0747e-02 \\
0.6 & 2.5
& 2.6810  &  7.2388e-02 
& 2.5278  & 1.1126e-02 
& 2.4794  &  8.2411e-03\\
0.75 & 4.0
& 7.3156  & 8.2891e-01
& 4.7208  & 1.8020e-01 
& 4.1591  &  3.9767e-02\\
\bottomrule
\end{tabular}
\end{table}
\vspace{-8pt}
\begin{table}[htbp]
\centering
\caption{Comparison of exact and numerical values of $\mathcal{M}_{1}(t)=1$.}
\label{tab:2m1}
\renewcommand{\arraystretch}{1.1}
\setlength{\tabcolsep}{5pt}

\begin{tabular}{c c cc cc cc}
\toprule
$t$ & Exact 
& \multicolumn{2}{c}{$80\quad grids$}
& \multicolumn{2}{c}{$160\quad grids$}
& \multicolumn{2}{c}{$320\quad grids$} \\
\cmidrule(lr){3-4} \cmidrule(lr){5-6} \cmidrule(lr){7-8}
& 
& Num & Rel. Err.
& Num & Rel. Err.
& Num & Rel. Err. \\
\midrule
0.15 & 1.0
& 0.9598 &  4.0162e-02
& 0.9596   & 4.0430e-02
& 0.9595  & 4.0471e-02 \\
0.3 & 1.0
& 0.9601  & 3.9899e-02
& 0.9596    &  4.0410e-02
&  0.9595   & 4.0517e-02 \\
0.45 & 1.0
&  0.9636 & 3.6355e-02
& 0.9599  & 4.0127e-02
& 0.9595   & 4.0510e-02\\
0.6 & 1.0
& 0.9670  & 3.3028e-02
& 0.9612    &   3.8834e-02
& 0.9597  & 4.0280e-02\\
0.75 & 1.0
& 1.0333   &  3.3329e-02
& 0.9707  & 2.9326e-02 
& 0.9616   &  3.8415e-02\\
\bottomrule
\end{tabular}
\end{table}

\begin{figure}[htbp]
\centering
\begin{subfigure}{0.32\textwidth}
    \centering
    \includegraphics[width=\linewidth]{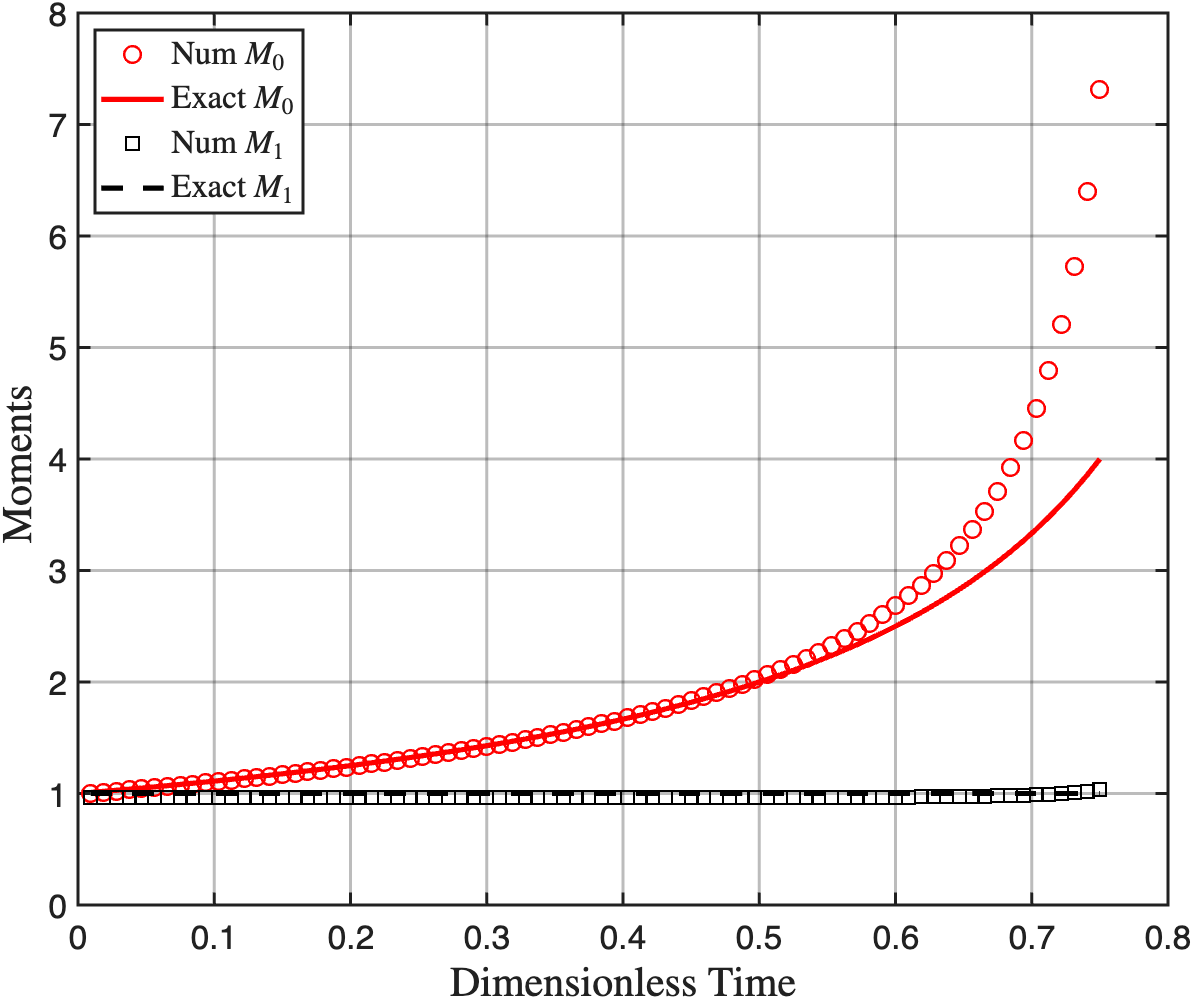}
    \caption{$80$ grids}
\end{subfigure}
\hfill
\begin{subfigure}{0.32\textwidth}
    \centering
    \includegraphics[width=\linewidth]{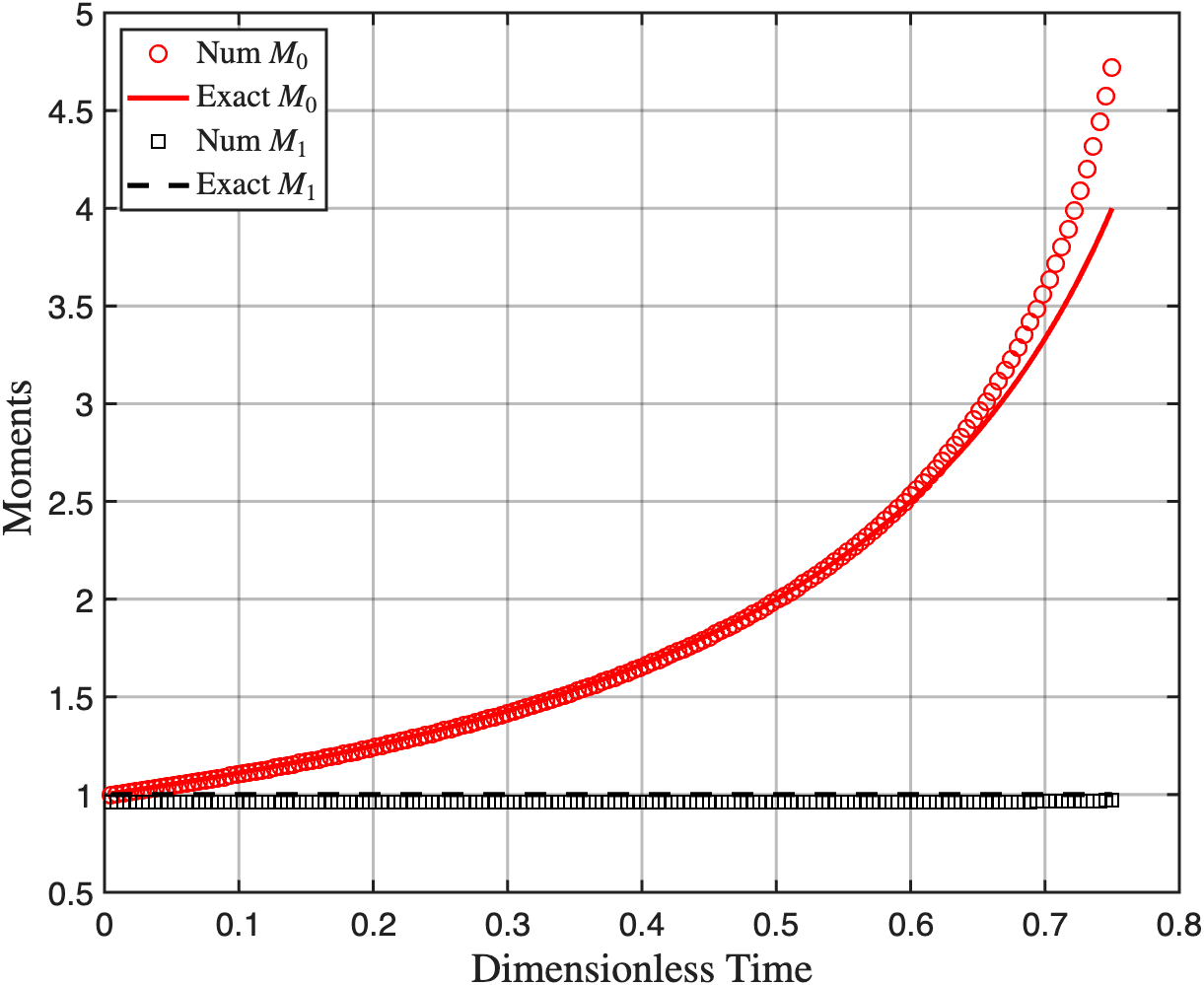}
    \caption{$160$ grids}
\end{subfigure}
\hfill
\begin{subfigure}{0.32\textwidth}
    \centering
    \includegraphics[width=\linewidth]{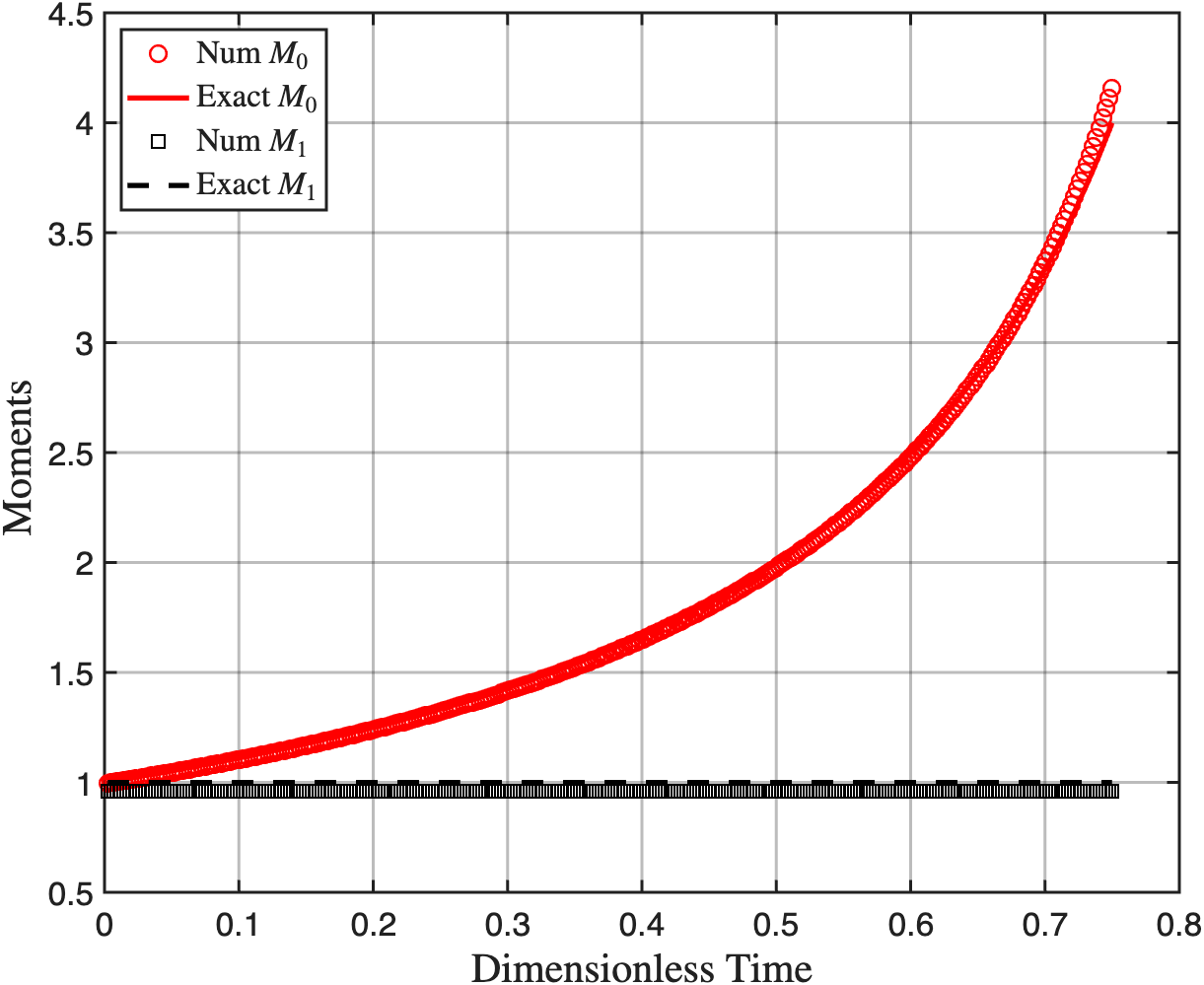}
    \caption{$320$ grids}
\end{subfigure}

\caption{Comparison of Test Case 2 for different grid resolutions.}
\label{fig:case2}
\end{figure}
\vspace{0.2cm}

\textbf{Test Case 3: Product Kernel with breakage kernel $\beta(x,y;z)=\frac{3}{2}x^{1/2}y^{1/2}$}

Now let us verify the proposed method by considering the ternary breakage kernel $\beta(x,y;z)=\frac{3}{2}x^{1/2}y^{1/2}$ with the product collisional kernel $\Gamma(x,y)=xy$. Here we consider the initial condition as $u(x,0)=e^{-x} $ with moments as $\mathcal{M}_0(t)=1+t$ and $\mathcal{M}_1(t)=1$. The results for Case 3 are presented in Figure~\ref{fig:case3}. It is observed that mass is conserved throughout the simulation. Moreover, the accuracy of the zeroth moment improves with grid refinement. The relative errors of the moments computed using the BDF2 scheme are reported in Table~\ref{tab:3m0}--Table~\ref{tab:3m1} for various grid resolutions. It is evident that the relative errors in all moments decrease as the grid is refined.
\begin{table}[!h]
\centering
\caption{Comparison of exact and numerical values of $\mathcal{M}_{0}(t)=1+t$.}
\label{tab:3m0}
\renewcommand{\arraystretch}{1.1}
\setlength{\tabcolsep}{5pt}

\begin{tabular}{c c cc cc cc}
\toprule
$t$ & Exact 
& \multicolumn{2}{c}{$80\quad grids$}
& \multicolumn{2}{c}{$160\quad grids$}
& \multicolumn{2}{c}{$320\quad grids$} \\
\cmidrule(lr){3-4} \cmidrule(lr){5-6} \cmidrule(lr){7-8}
& 
& Num & Rel. Err.
& Num & Rel. Err.
& Num & Rel. Err. \\
\midrule
1.0 & 2.0
& 2.0022  & 1.1187e-03 
& 2.0005 & 2.4559e-04  
& 2.0001 & 2.7344e-05\\
2.0 & 3.0
& 3.0041 & 1.3664e-03
& 3.0010 & 3.4126e-04 
& 3.0003 & 8.5073e-05\\
3.0 & 4.0
& 4.0061 & 1.5290e-03
& 4.0015  & 3.8206e-04
& 4.0004  & 9.5491e-05  \\
4.0 & 5.0
& 5.0083  & 1.6678e-03
& 5.0021 &  4.1663e-04 
& 5.0005 &  1.0412e-04 \\
5.0 & 6.0
& 6.0107  & 1.7903e-03
& 6.0027  & 4.4710e-04
& 6.0007   & 1.1171e-04  \\
\bottomrule
\end{tabular}
\end{table}
\vspace{-8pt}
\begin{table}[!h]
\centering
\caption{Comparison of exact and numerical values of $\mathcal{M}_{1}(t)=1$.}
\label{tab:3m1}
\renewcommand{\arraystretch}{1.1}
\setlength{\tabcolsep}{5pt}
\begin{tabular}{c c cc cc cc}
\toprule
$t$ & Exact 
& \multicolumn{2}{c}{$80\quad grids$}
& \multicolumn{2}{c}{$160\quad grids$}
& \multicolumn{2}{c}{$320\quad grids$} \\
\cmidrule(lr){3-4} \cmidrule(lr){5-6} \cmidrule(lr){7-8}
& 
& Num & Rel. Err.
& Num & Rel. Err.
& Num & Rel. Err. \\
\midrule
1.0 & 1
& 1.0025   & 2.4796e-03
& 1.0002 & 2.4480e-04
& 0.9997 & 3.1339e-04\\
2.0 & 1
& 1.0034  & 3.4059e-03
& 1.0008  & 8.4613e-04 
& 1.0002 & 2.0768e-04\\
3.0 & 1
& 1.0037  & 3.6501e-03
&  1.0009 & 9.0809e-04
&  1.0002  & 2.2646e-04\\
4.0 & 1
& 1.0038  & 3.7598e-03
& 1.0009  & 9.3047e-04
& 1.0002  & 2.3141e-04\\
5.0 & 1
& 1.0038  & 3.7651e-03
& 1.0009  &  9.2367e-04 
& 1.0002   &  2.2862e-04\\
\bottomrule
\end{tabular}
\end{table}
\begin{figure}[!h]
\centering
\begin{subfigure}{0.32\textwidth}
    \centering
    \includegraphics[width=\linewidth]{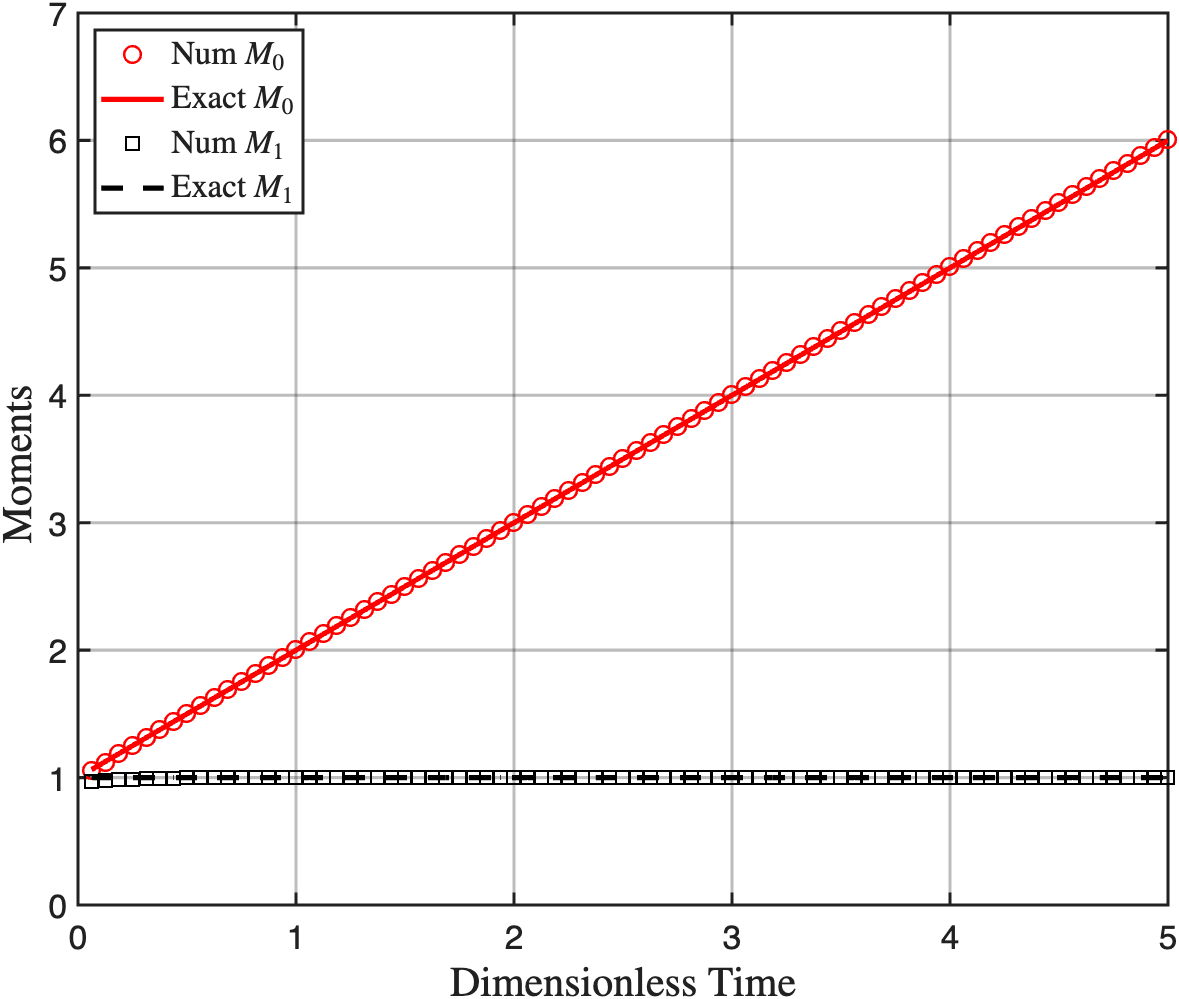}
    \caption{$80$ grids}
\end{subfigure}
\hfill
\begin{subfigure}{0.32\textwidth}
    \centering
    \includegraphics[width=\linewidth]{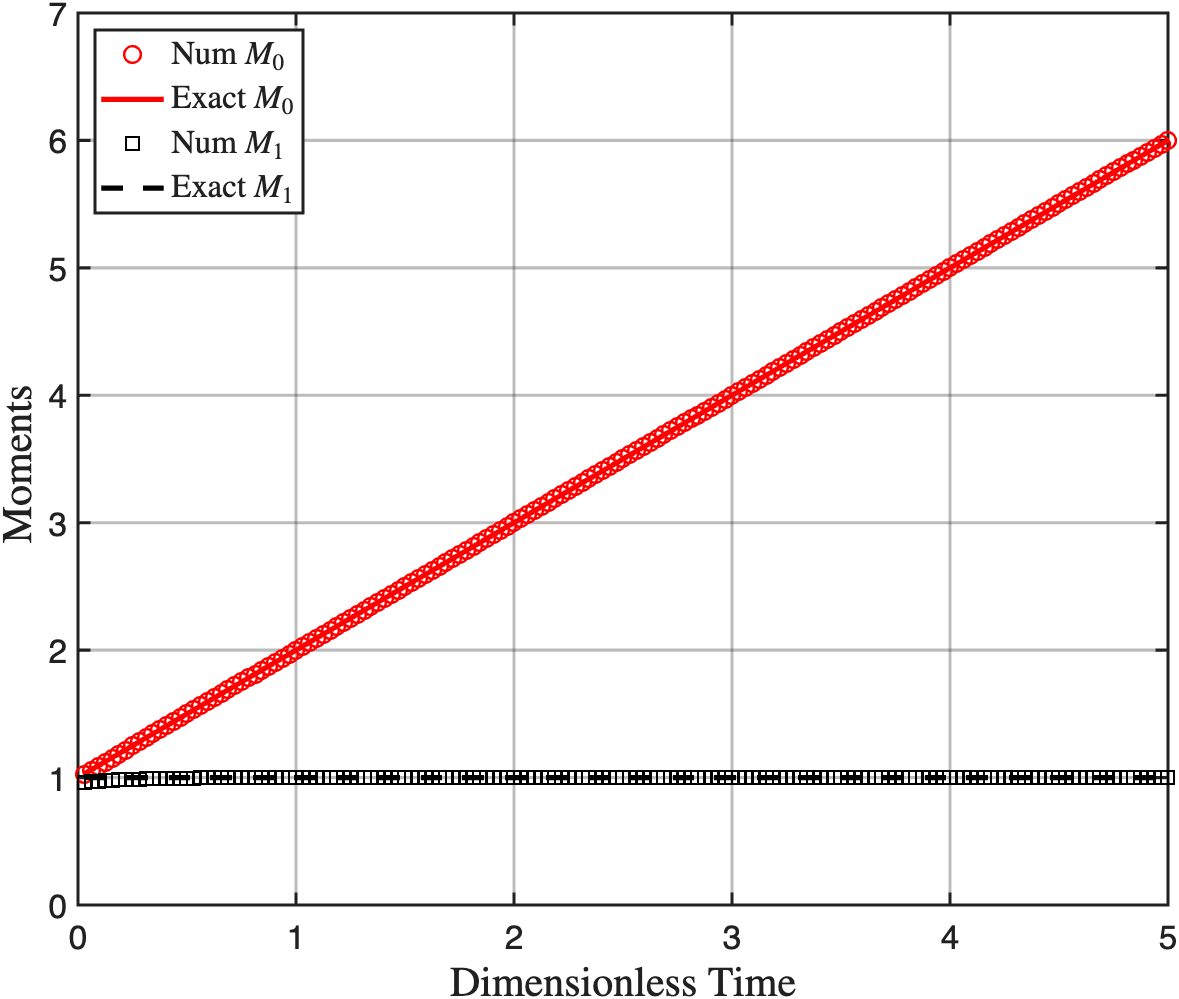}
    \caption{$160$ grids}
\end{subfigure}
\hfill
\begin{subfigure}{0.32\textwidth}
    \centering
    \includegraphics[width=\linewidth]{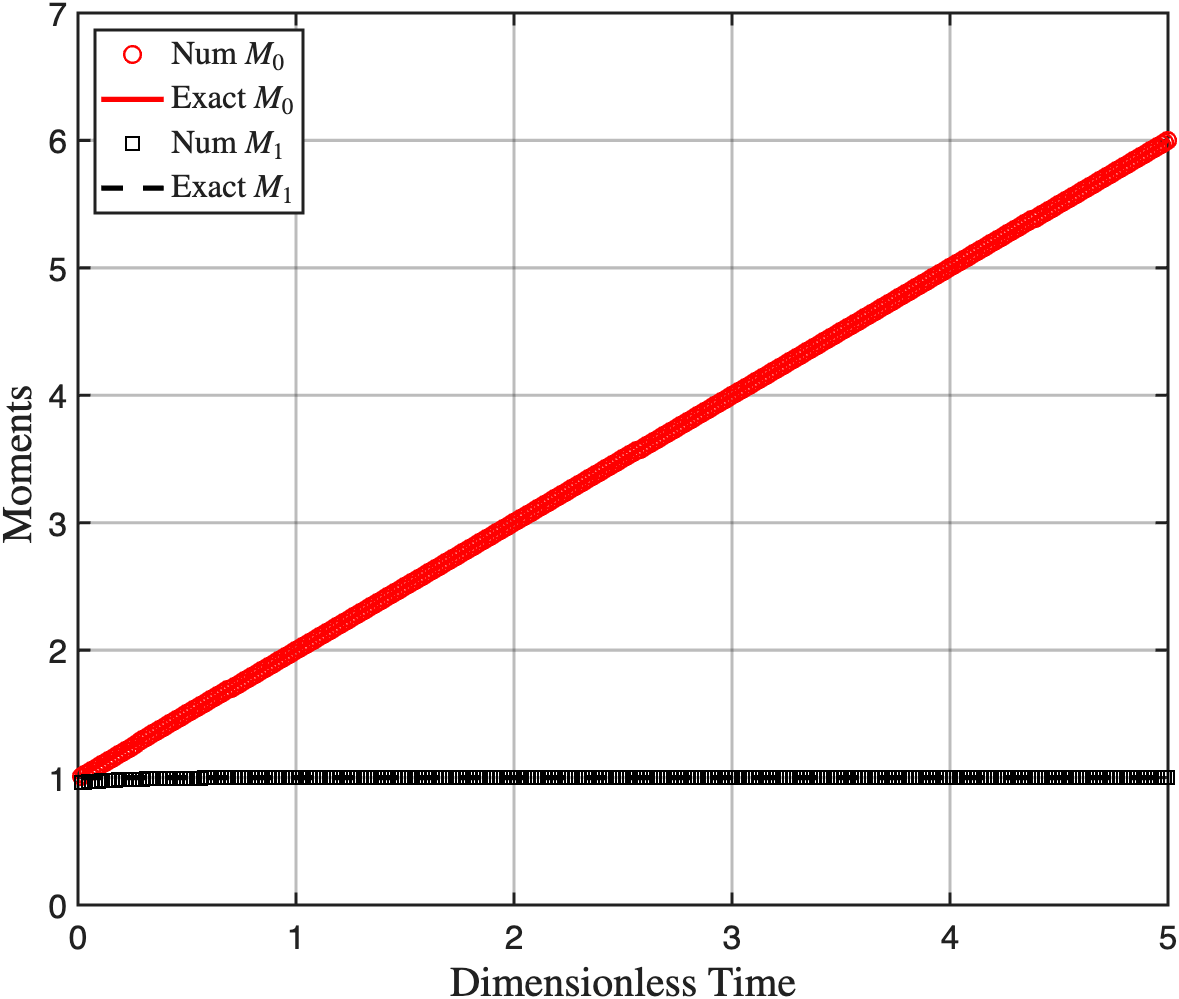}
    \caption{$320$ grids}
\end{subfigure}
\caption{Comparison of Test Case 3 for different grid resolutions.}
\label{fig:case3}
\end{figure}
\vspace{0.3cm}

\textbf{Test Case 4: Polymerization kernel with an exponential initial distribution}

In this example, we consider the polymerization collisional kernel $\Gamma(x,y)=(x+c)^{1/3}(y+c)^{1/3}$ with c=0 and random binary fragmentation rate. Initial condition is taken as  $u(x,0)=e^{-x} $ with moments as $\mathcal{M}_0(t)=1+\frac{49}{50}t$ and $\mathcal{M}_1(t)=1$. The results for Case 4 are illustrated in Figure~\ref{fig:case4}. It is observed that mass is conserved throughout the simulation. The accuracy of the zeroth moment improves as the grid is refined. The relative errors of the moments computed using the BDF2 scheme are listed in Table~\ref{tab:4m0} - Table~\ref{tab:4m1} for different grid resolutions. The proposed scheme conserves both the number of particles and the system volume and is in good agreement with the exact values.

\begin{table}[!h]
\centering
\caption{Comparison of exact and numerical values of $\mathcal{M}_{0}(t)=1+\frac{49t}{50}$.}
\label{tab:4m0}
\renewcommand{\arraystretch}{1.1}
\setlength{\tabcolsep}{5pt}

\begin{tabular}{c c cc cc cc}
\toprule
$t$ & Exact 
& \multicolumn{2}{c}{$80\quad grids$}
& \multicolumn{2}{c}{$160\quad grids$}
& \multicolumn{2}{c}{$320\quad grids$} \\

\cmidrule(lr){3-4} \cmidrule(lr){5-6} \cmidrule(lr){7-8}

& 
& Num & Rel. Err.
& Num & Rel. Err.
& Num & Rel. Err. \\
\midrule

1.0 & 1.9800
& 2.0029  & 1.1584e-02 
& 2.0007  & 1.0437e-02 
& 2.0001 & 1.0151e-02\\

2.0 & 2.9600
& 3.0052 & 1.5268e-02
& 3.0013 & 1.3952e-02  
& 3.0003 & 1.3623e-02\\

3.0 & 3.9400
& 4.0076 & 1.7159e-02
& 4.0019  & 1.5711e-02
& 4.0005  & 1.5349e-02   \\

4.0 & 4.9200
& 5.0102  & 1.8337e-02
& 5.0026 &  1.6779e-02
& 5.0006 &  1.6390e-02  \\

5.0 & 5.9000
&  6.0130 & 1.9154e-02
& 6.0032 & 1.7500e-02
& 6.0008    & 1.7087e-02  \\

\bottomrule
\end{tabular}
\end{table}

\begin{table}[!h]
\centering
\caption{Comparison of exact and numerical values of $\mathcal{M}_{1}(t)=1$.}
\label{tab:4m1}
\renewcommand{\arraystretch}{1.1}
\setlength{\tabcolsep}{5pt}

\begin{tabular}{c c cc cc cc}
\toprule
$t$ & Exact 
& \multicolumn{2}{c}{$80\quad grids$}
& \multicolumn{2}{c}{$160\quad grids$}
& \multicolumn{2}{c}{$320\quad grids$} \\

\cmidrule(lr){3-4} \cmidrule(lr){5-6} \cmidrule(lr){7-8}

& 
& Num & Rel. Err.
& Num & Rel. Err.
& Num & Rel. Err. \\
\midrule

1.0 & 1
& 1.0035   &  3.4855e-03
& 1.0005  & 4.9590e-04
& 0.9997 & 2.5064e-04\\

2.0 & 1
& 1.0045  & 4.4717e-03
& 1.0011  & 1.1122e-03
& 1.0003  & 2.7416e-04\\

3.0 & 1
& 1.0047  & 4.7470e-03
& 1.0012 & 1.1819e-03
& 1.0003  & 2.9488e-04\\

4.0 & 1
& 1.0049   & 4.8763e-03
& 1.0012   & 1.2092e-03
& 1.0003  & 3.0108e-04\\

5.0 & 1
& 1.0049  & 4.8938e-03
& 1.0012  &  1.2056e-03
& 1.0003   &  2.9908e-04\\

\bottomrule
\end{tabular}
\end{table}
\begin{figure}[!h]
\centering
\begin{subfigure}{0.32\textwidth}
    \centering
    \includegraphics[width=\linewidth]{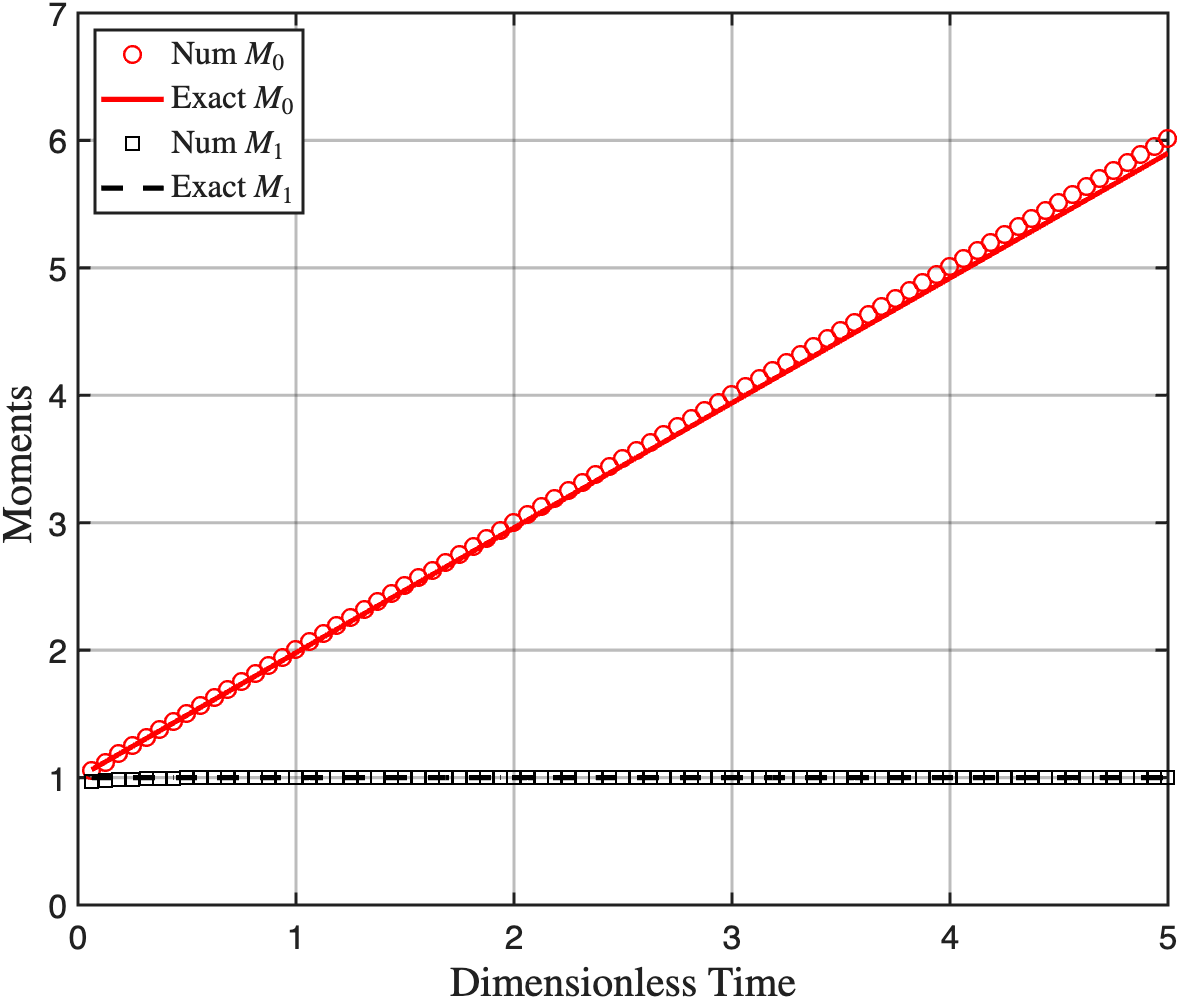}
    \caption{$80$ grids}
\end{subfigure}
\hfill
\begin{subfigure}{0.32\textwidth}
    \centering
    \includegraphics[width=\linewidth]{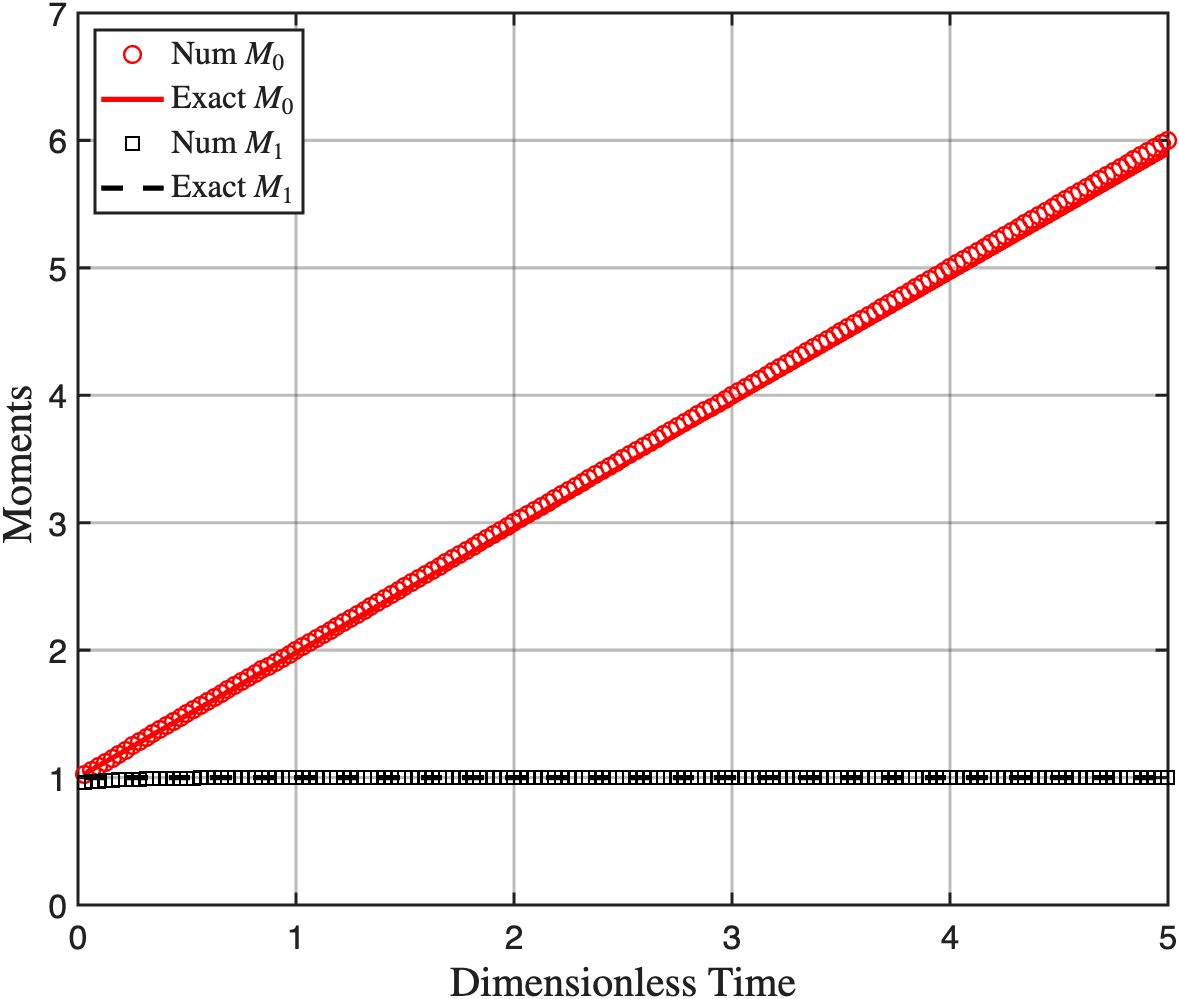}
    \caption{$160$ grids}
\end{subfigure}
\hfill
\begin{subfigure}{0.32\textwidth}
    \centering
    \includegraphics[width=\linewidth]{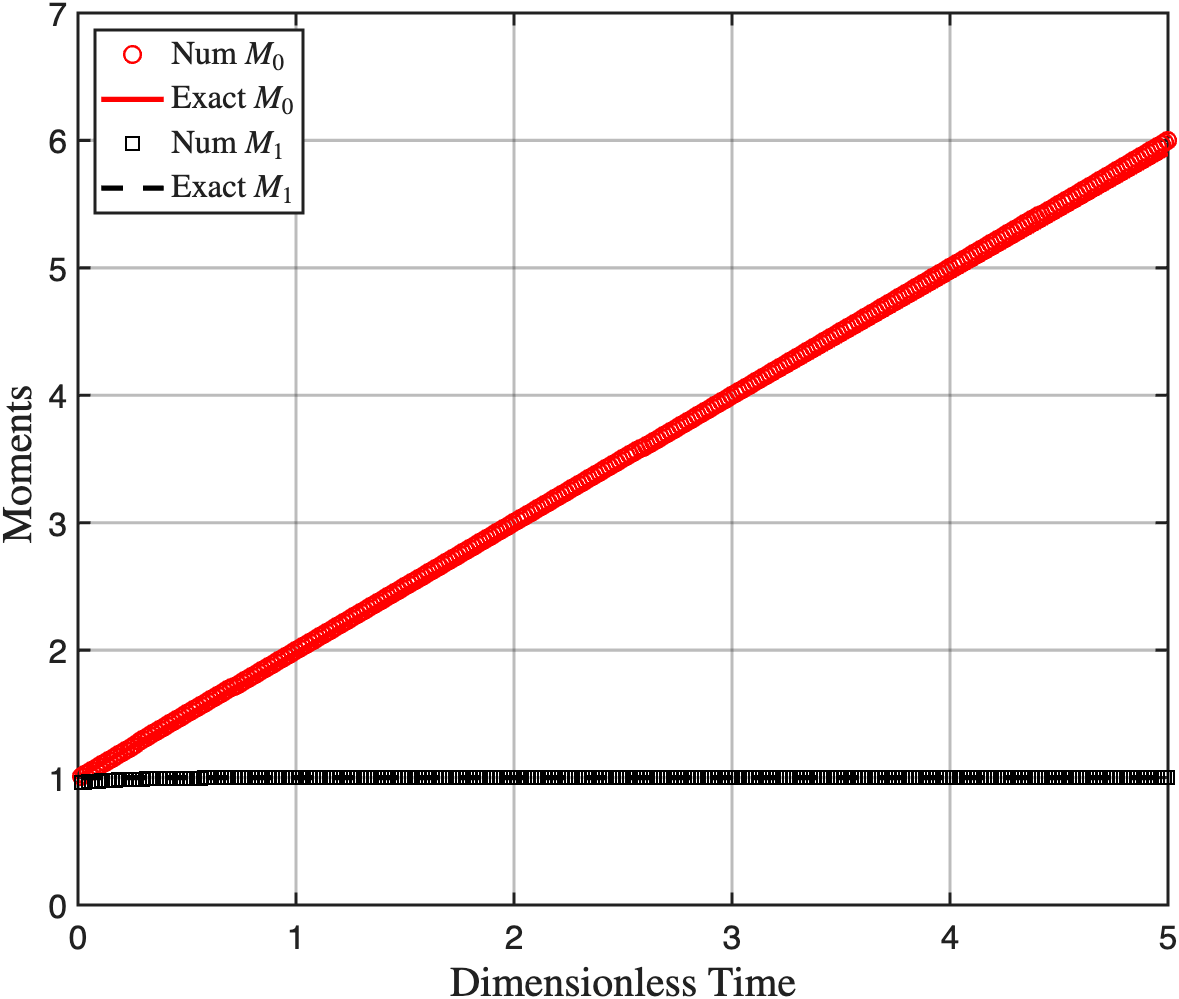}
    \caption{$320$ grids}
\end{subfigure}

\caption{Comparison of Test Case 4 for different grid resolutions.}
\label{fig:case4}
\end{figure}
\subsubsection{Two Dimensional Case}
\textbf{Test Case 5: Product kernel with mono-dispersed initial condition}

For this case, we consider a two-dimensional breakage mechanism in which two distinct properties describe particles. The particle collision rate is assumed to vary linearly with the hypervolume of the particles. Additionally, upon breakage, particles fragment randomly along the property axes, resulting in four daughter fragments. The corresponding breakage kernels for this test case is $\beta(\mathbf{x},\mathbf{y};\mathbf{z})=\frac{4}{y_1y_2}$ with the product collisional kernel $\Gamma(\mathbf{x},\mathbf{y})=x_1x_2y_1y_2$. Here we consider the  mono-dispersed initial conditions $u(\mathbf{x},0)=\delta(x_1-1)\delta(x_2-1) $ with moments as $\mathcal{M}_{00}(t)=1+3t$ and $\mathcal{M}_{11}(t)=1$. The simulations are carried out over the time interval $t\in [0,3]$ . The numerical results are summarized in Tables~\ref{tab:2D_1}–\ref{tab:2D_11} for different grid sizes, while the corresponding solution profiles for Case 5 are illustrated in Figure~\ref{fig:2D_1}. The scheme maintains mass conservation throughout the simulation.
 \begin{table}[!h]
\centering
\caption{Comparison of exact and numerical values of $\mathcal{M}_{00}(t)=1+3t$ for Test Case 5}
\label{tab:2D_1}
\renewcommand{\arraystretch}{1.1}
\setlength{\tabcolsep}{5pt}

\begin{tabular}{c c cc cc cc}
\toprule
$t$ & Exact
& \multicolumn{2}{c}{$80\times80$}
& \multicolumn{2}{c}{$120\times120$}
& \multicolumn{2}{c}{$160\times160$} \\
\cmidrule(lr){3-4} \cmidrule(lr){5-6} \cmidrule(lr){7-8}
& 
& Num & Rel. Err.
& Num & Rel. Err.
& Num & Rel. Err. \\
\midrule

0.6 & 2.8
& 2.82284 & 8.16e-3
& 2.81821 & 6.5e-3
& 2.81675 & 5.98e-3 \\

1.2 & 4.6
& 4.6746 & 1.62e-2
& 4.66369 & 1.38e-2
& 4.66034 & 1.31e-2 \\

1.8 & 6.4
& 6.54567 & 2.28e-2
& 6.52703 & 1.98e-2
& 6.52139 & 1.89e-2 \\

2.4 & 8.2
& 8.43004 & 2.81e-2
& 8.40241 & 2.46e-2
& 8.39413 & 2.37e-2 \\

3.0 & 10
& 10.324 & 3.24e-2
& 10.2862 & 2.86e-e
& 10.275 & 2.75e-2 \\

\bottomrule
\end{tabular}
\end{table}
\begin{table}[!h]
\centering
\caption{Comparison of exact and numerical values of $\mathcal{M}_{11}(t)=1$ for Test Case 5.}
\label{tab:2D_11}
\renewcommand{\arraystretch}{1.1}
\setlength{\tabcolsep}{5pt}

\begin{tabular}{c c cc cc cc}
\toprule
$t$ & Exact
& \multicolumn{2}{c}{$20\times20$}
& \multicolumn{2}{c}{$40\times40$}
& \multicolumn{2}{c}{$60\times60$} \\
\cmidrule(lr){3-4} \cmidrule(lr){5-6} \cmidrule(lr){7-8}
& 
& Num & Rel. Err.
& Num & Rel. Err.
& Num & Rel. Err. \\
\midrule

0.6 & 1
& 1.0113 & 1.13e-2
& 1.00989 & 9.89e-3
& 1.00944 & 9.44e-3 \\

1.2 & 1
& 1.0175 & 1.75e-2
& 1.01572 & 1.57e-2
& 1.01517 & 1.52e-2 \\

1.8 & 1
& 1.0217 & 2.17e-2
& 1.01957 & 1.96e-2
& 1.01893 & 1.89e-2 \\

2.4 & 1
& 1.02462 & 2.46e-2
& 1.02222 & 2.22e-2
& 1.0215 & 2.15e-2 \\

3.0 & 1
& 1.02616 & 2.62e-2
& 1.02411 & 2.41e-2
& 1.02332 & 2.33e-2 \\

\bottomrule
\end{tabular}
\end{table}

\begin{figure}[H]
\centering

\begin{subfigure}{0.32\textwidth}
    \centering
    \includegraphics[width=\linewidth]{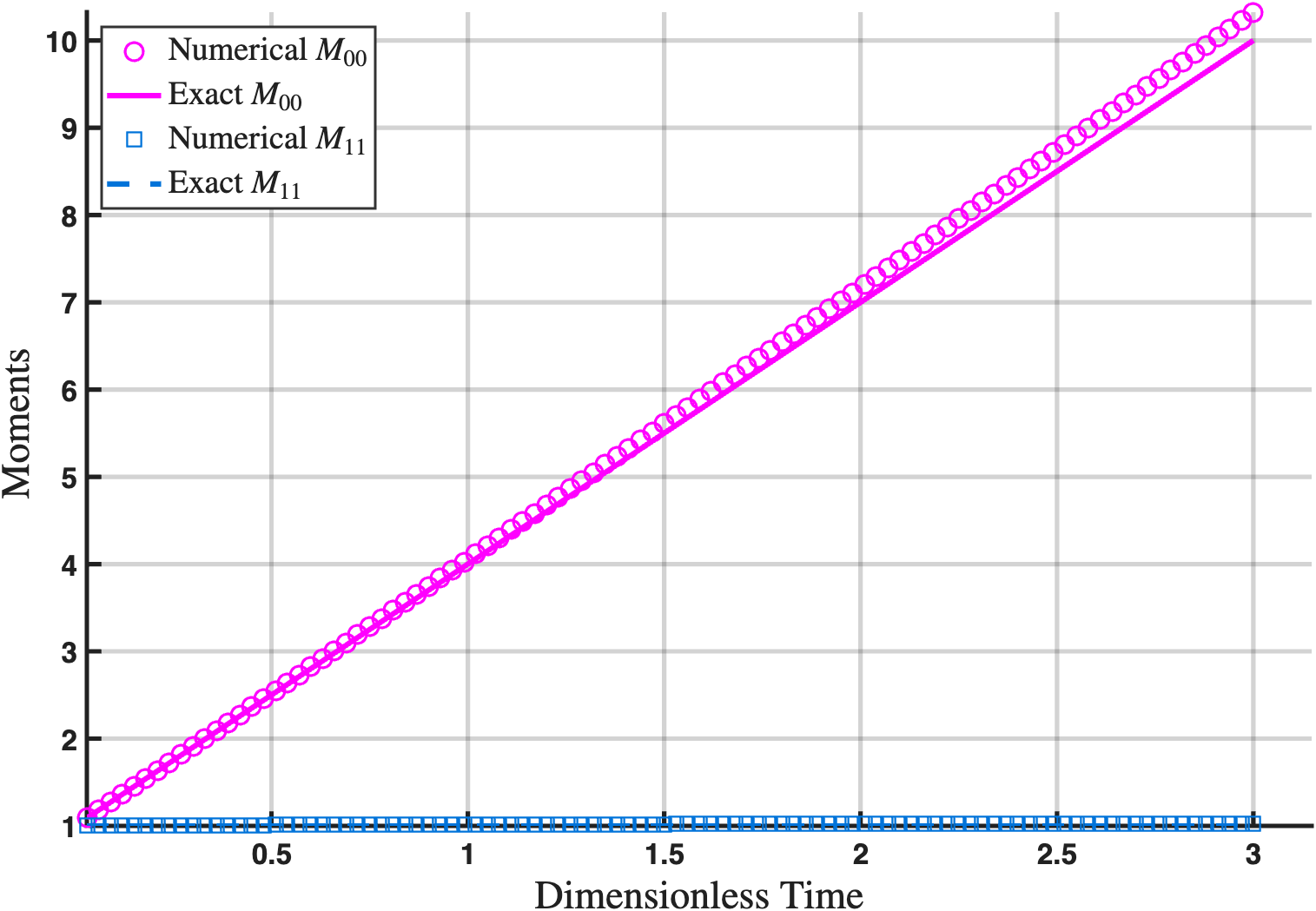}
    \caption{$80\times80$ grid}
\end{subfigure}
\hfill
\begin{subfigure}{0.32\textwidth}
    \centering
    \includegraphics[width=\linewidth]{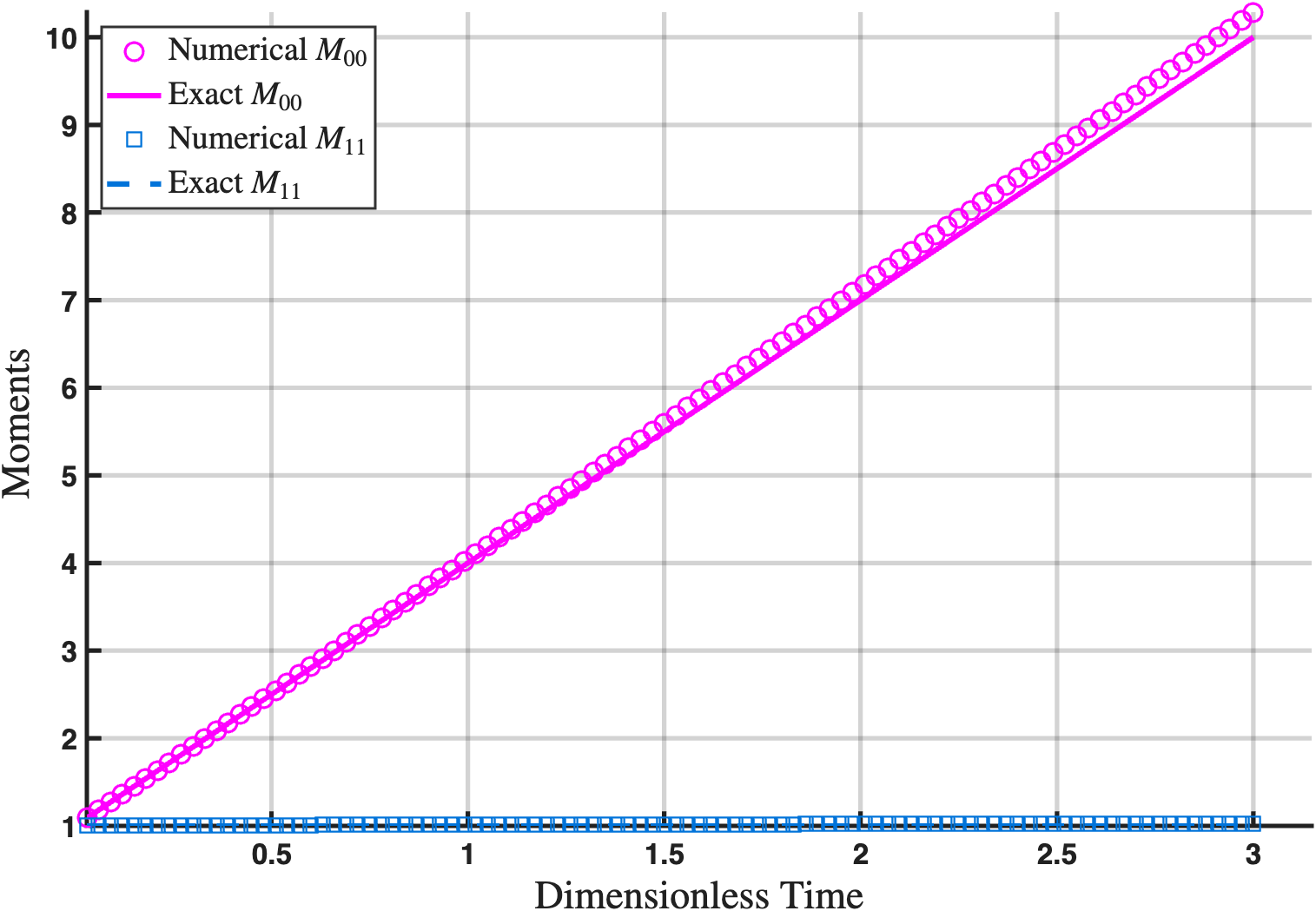}
    \caption{$120\times120$ grid}
\end{subfigure}
\hfill
\begin{subfigure}{0.32\textwidth}
    \centering
    \includegraphics[width=\linewidth]{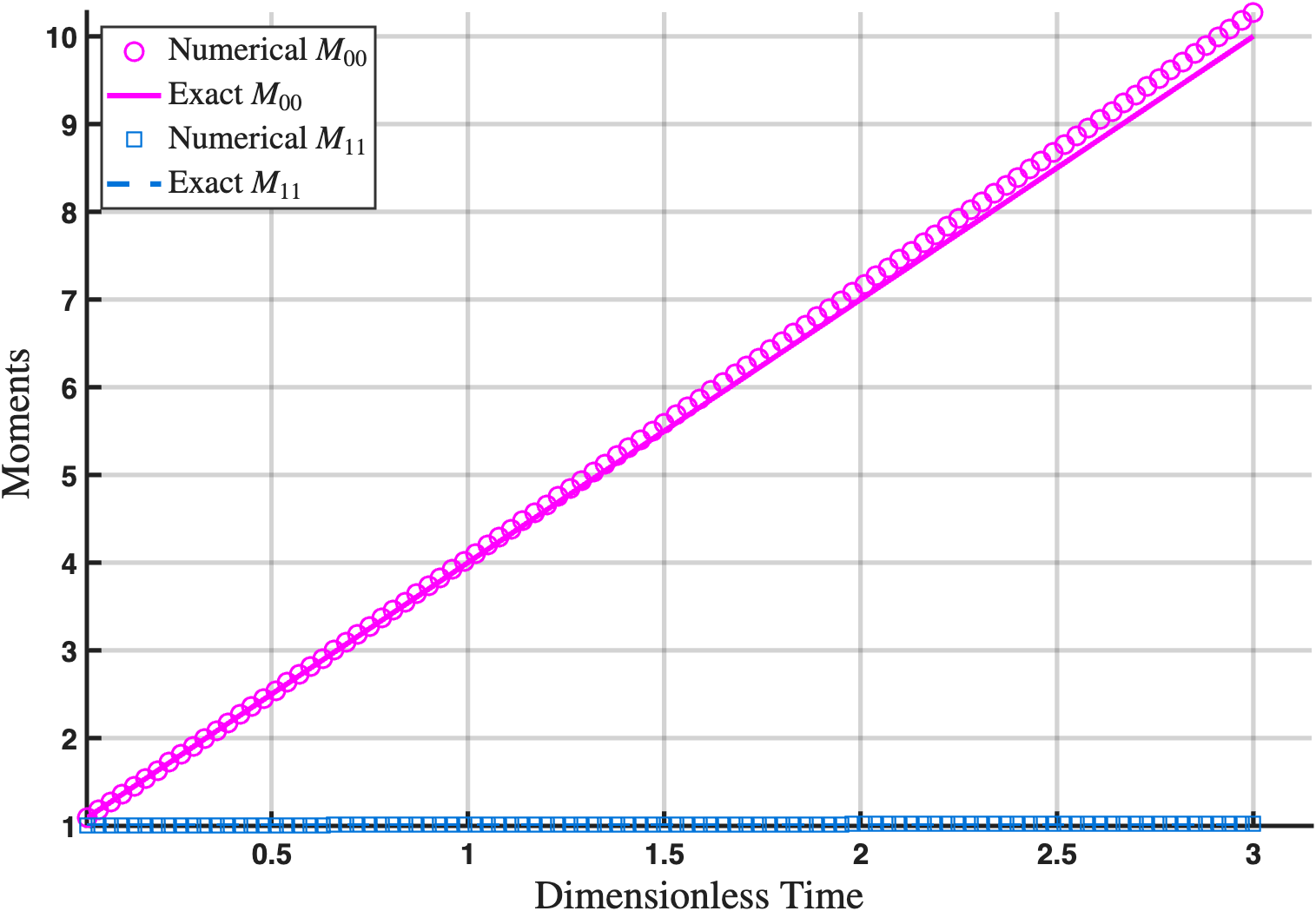}
    \caption{$160\times160$ grid}
\end{subfigure}
\caption{Comparison of Test Case 5 for different grid resolutions.}
\label{fig:2D_1}
\end{figure}
\subsubsection{Three Dimensional Cases}
\textbf{Test Case 6: Product kernel with mono-dispersed initial condition}

In this 3D test case, particles are characterized by three internal properties. Here, the fragmentation mechanism involves the simultaneous random cleavage of both colliding particles along every property axis, resulting in a shatter effect that produces eight daughter fragments.The breakage kernel for this test case is given $\beta(\mathbf{x},\mathbf{y},\mathbf{z})=\frac{8}{y_1y_2y_3}$, while the product-type collisional kernel is defined as $\Gamma(\mathbf{x},\mathbf{y})=x_1x_2x_3y_1y_2y_3$. We consider mono-disperse initial conditions of the form $u(\mathbf{x},0)=\delta(x_1,0)\delta(x_2,0)\delta(x_3,0)$, with the corresponding moments $\mathcal{M}_{000}(t)=1+7t$ and $\mathcal{M}_{111}(t)=1$. The simulations are performed over the time interval $t \in [0,2]$. The numerical results for different grid sizes are presented in Tables~\ref{tab:3D_1}–\ref{tab:3D_11}, while the associated solution profiles for Case 6 are shown in Figure~\ref{fig:3D_1}. The scheme consistently preserves hypervolume throughout the simulation.

\begin{table}[H]
\centering
\caption{Comparison of exact and numerical values of $\mathcal{M}_{000}(t)=1+7t$ for Test Case 6}
\label{tab:3D_1}
\renewcommand{\arraystretch}{1.1}
\setlength{\tabcolsep}{5pt}

\begin{tabular}{c c cc cc cc}
\toprule
$t$ & Exact
& \multicolumn{2}{c}{$15\times15$}
& \multicolumn{2}{c}{$20\times20$}
& \multicolumn{2}{c}{$25\times25$} \\
\cmidrule(lr){3-4} \cmidrule(lr){5-6} \cmidrule(lr){7-8}
& 
& Num & Rel.~Err.
& Num & Rel.~Err.
& Num & Rel.~Err. \\
\midrule

0.4 & 3.8
& 3.89905 & 2.62e-2
& 3.8675 & 1.78e-2
& 3.85306 & 1.39e-2 \\
0.8 & 6.6
& 6.9867 & 5.85e-2
& 6.83627& 3.57e-2
& 6.7907 & 2.89e-2 \\

1.2 & 9.4
& 10.0945 & 7.38e-2
& 9.88797 & 5.19e-2
& 9.79495 & 4.2e-2 \\

1.6& 12.2
& 13.3592 & 9.5e-2
& 13.0095 & 6.64e-2
& 12.853 & 5.35e-2 \\

2.0 & 15
& 16.722 & 1.15e-1
& 16.1918 & 7.94e-2
& 15.9563 & 6.37e-2 \\
\bottomrule
\end{tabular}
\end{table}

\begin{table}[H]
\centering
\caption{Comparison of exact and numerical values of $\mathcal{M}_{111}(t)=1$ for Test Case 6.}
\label{tab:3D_11}
\renewcommand{\arraystretch}{1.1}
\setlength{\tabcolsep}{5pt}
\begin{tabular}{c c cc cc cc}
\toprule
$t$ & Exact
& \multicolumn{2}{c}{$15\times15$}
& \multicolumn{2}{c}{$20\times20$}
& \multicolumn{2}{c}{$25\times25$} \\
\cmidrule(lr){3-4} \cmidrule(lr){5-6} \cmidrule(lr){7-8}
& 
& Num & Rel. Err.
& Num & Rel. Err.
& Num & Rel. Err. \\
\midrule
0.4 & 1
& 1.02564 & 2.56e-2
& 1.01939 & 1.94e-2
& 1.01652 & 1.65e-2 \\
0.8 & 1
& 1.04262 & 4.26e-2
& 1.03228 & 3.22e-2
& 1.02758  & 2.76e-2 \\
1.2 & 1
& 1.05674 & 5.67e-2
& 1.04247 & 4.25e-2
& 1.03602 & 3.6e-2 \\
1.6 & 1
& 1.06895 & 6.89e-2
& 1.05086 & 5.09e-2
& 1.04272 & 4.27e-2 \\
2.0 & 1
& 1.07989 & 7.98e-2
& 1.05803 & 5.8e-2
& 1.04825 & 4.82e-2 \\
\bottomrule
\end{tabular}
\end{table}

\begin{figure}[H]
\centering
\begin{subfigure}{0.32\textwidth}
    \centering
    \includegraphics[width=\linewidth]{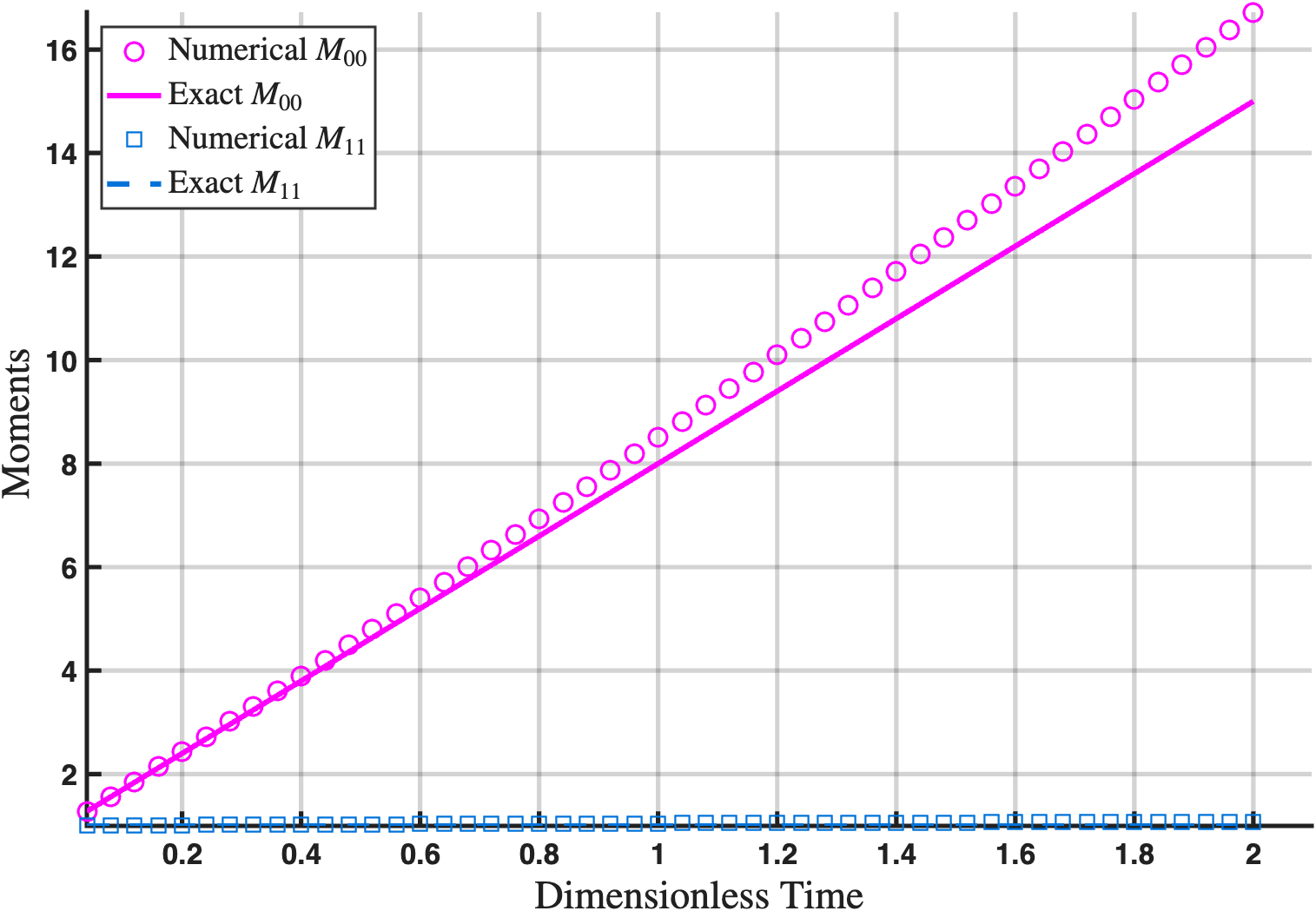}
    \caption{$15\times15$ grid}
\end{subfigure}
\hfill
\begin{subfigure}{0.32\textwidth}
    \centering
    \includegraphics[width=\linewidth]{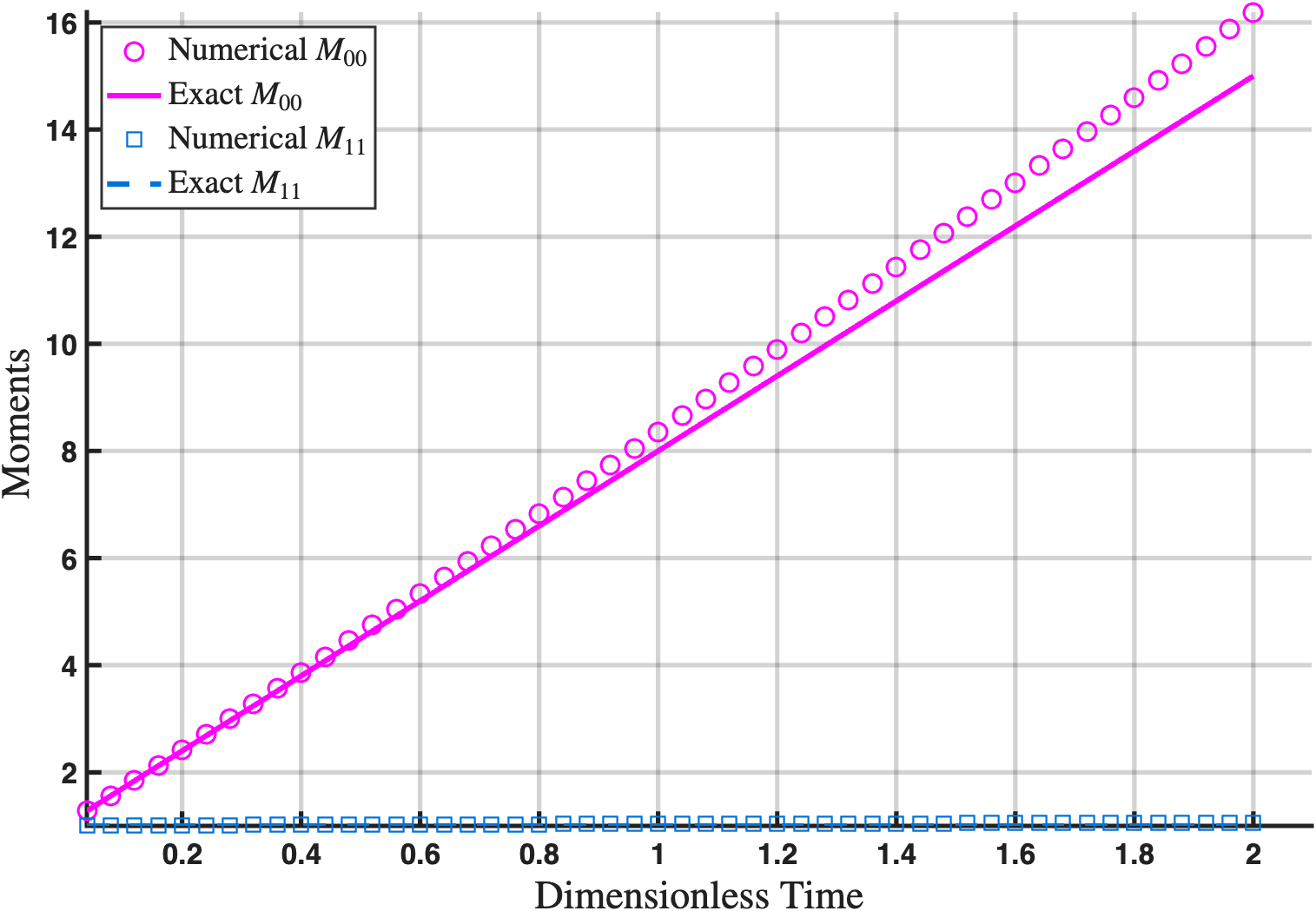}
    \caption{$20\times20$ grid}
\end{subfigure}
\hfill
\begin{subfigure}{0.32\textwidth}
    \centering
    \includegraphics[width=\linewidth]{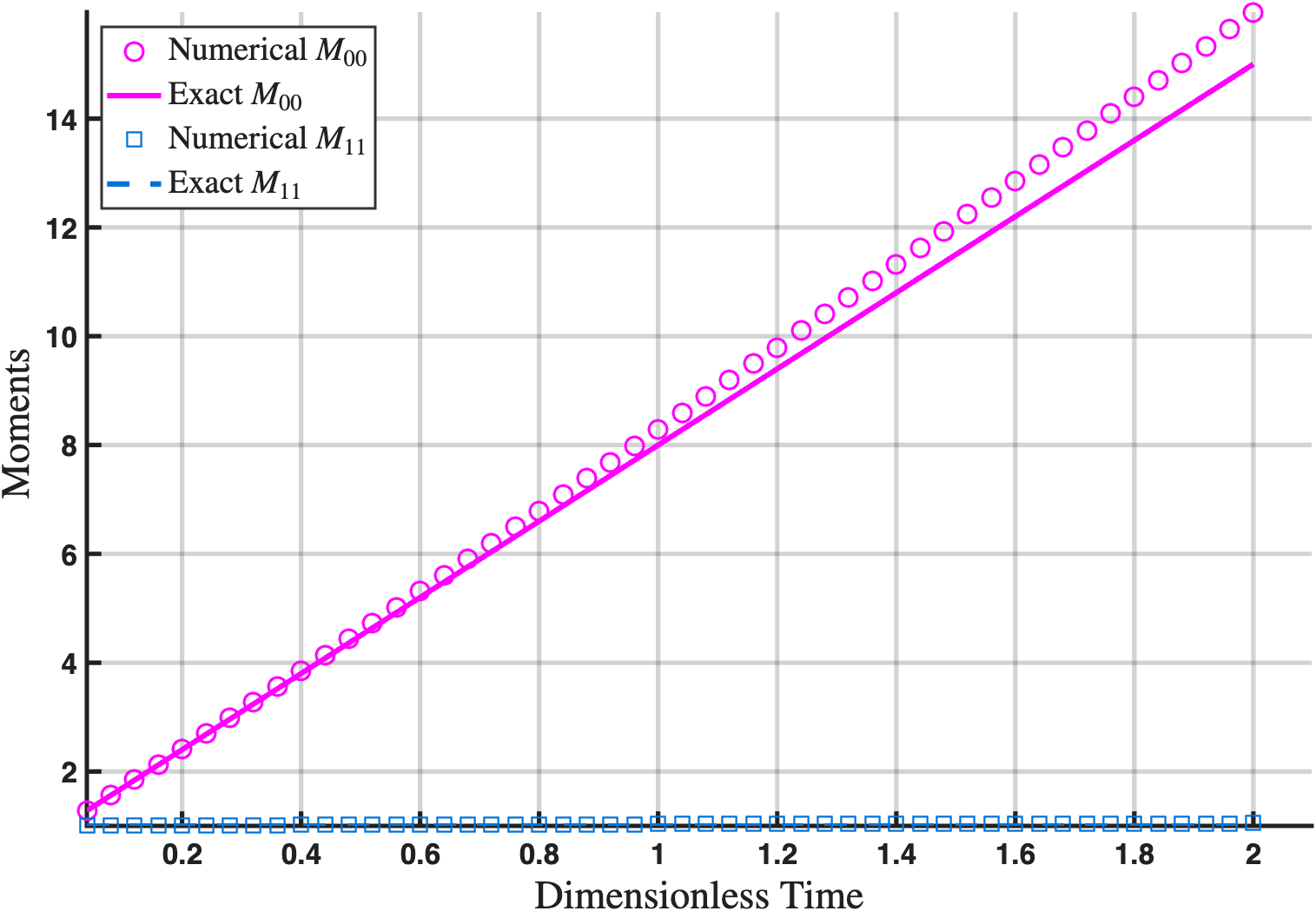}
    \caption{$25\times25$ grid}
\end{subfigure}

\caption{Comparison of Test Case 6 for different grid resolutions.}
\label{fig:3D_1}
\end{figure}
\subsection{Rate of Convergence Analysis}
To assess the accuracy and reliability of the proposed numerical framework, we conduct a detailed rate-of-convergence study. The experimental orders of convergence (EOCs) are computed and compared with theoretical predictions to validate the method's effectiveness.
\subsubsection{One-Dimensional Test Cases}

\noindent
\textbf{Test Case 1 (Product Kernel with Exponential Initial Distribution).} 

Consider the collisional kernel of the product $\Gamma(x,y)=xy$ and the binary breakage kernel $\beta(x,y;z)=\frac{2}{y}$. The exact solution of the density function of numbers is taken as $u(x,t)=(1+t)^2e^{-x(1+t)}$, which Kostoglou and Karabelas have derived \cite{kostoglou2000study}. For the analysis of $L^2(\mathcal{D}), H^1(\mathcal{D})$ and $L^\infty(\mathcal{D})$ errors and orders, we have used the domain as $\mathcal{D} := [10^{-9},\,5].$  The findings are
consistent with the theoretical analysis, presented in Table~\ref{tab:1D5}. In Table~\ref{tab:1D55} we compare the FEM solutions and absolute errors with the modified
variational iteration method (MVIM), variational iteration method (VIM), and Finite volume method (FVM) for different times and a fixed volume of 5. We observe that the FEM results agree well with the analytical solution at all times. Table~\ref{tab:1D555} compares the relative errors and Computation time of the FEM solution with the Homotopy perturbation method (HPM), the BLUES function, and the Semi-Analytical Approximation Method(APM). The errors in the number density functions at different times estimated by the present
approach are lower than those of VIM, MVIM, HPM, BLUES, and APM for the NCFE. 
This indicates the novelty of our proposed schemes.
\begin{table}[htbp]
\centering
\caption{Error table for NCBE (Test Case 1)}
\label{tab:1D5}
\renewcommand{\arraystretch}{1.1}
\begin{tabular}{c c c c c c c}
\hline
$N$ & $\|\mathcal{E}\|_{L^2(\Omega)}$ & EOC 
    & $\mathrm{RelError}_{L^\infty}$ & EOC
    & $\|\mathcal{E}\|_{H^1(\Omega)}$ & EOC \\

\hline
20  & 5.0971e-02 & - & 3.7500e-03 & - & 6.2539e-01 & - \\
40  & 1.2882e-02 & 1.98 & 9.3750e-04 & 2.00 & 3.0119e-01 & 1.05 \\
80  & 3.2293e-03 & 2.00 & 2.3437e-04  & 2.00 & 1.4752e-01 & 1.03 \\
160 & 8.0789e-04 & 2.00 & 5.8594e-05  & 2.00 & 7.2971e-02 & 1.02 \\
320 & 2.0201e-04 & 2.00 & 1.4648e-05 & 2.00 & 3.6286e-02 & 1.01 \\
\hline
\end{tabular}
\end{table}

\begin{table}[htbp]
\centering
\caption{Comparison of numerical solutions and absolute errors at $x=5$ for different time $t$ for test case 1.}
\label{tab:1D55}
\begin{tabular}{c|cccc|cccc}
\hline
 & \multicolumn{4}{c|}{Solution\cite{arora2026accurate}} & \multicolumn{4}{c}{Error} \\
\cline{2-9}
$t$ & Exact & VIM & MVIM & FEM & VIM & MVIM & FVM & FEM \\
\hline

0.3 & 0.00250 & 0.00250 & 0.00250 & 0.00250 & 1.855e-5 & 4.175e-6 & 9.756e-6 &  2.5319e-08\\

0.6 & 0.00085 & 0.00180 & 0.00105 & 0.00085 & 1.028e-3 & 2.003e-4 & 6.8776e-4 &   8.2486e-08\\

0.9 & 0.00027 & 0.01040 & 0.00194 & 0.00027 & 1.019e-2 & 1.670e-3 & 2.138e-3 &  1.6618e-07\\

1.2 & 8.084e-5 & 0.05000 & 0.00689 & 8.112e-5 & 5.032e-2 & 6.815e-3 & 7.485e-3& 2.8230e-07 \\

1.5 & 2.329e-5 & 0.17040 & 0.01880 & 2.373e-5 & 1.704e-1 & 1.886e-2 & 2.424e-2 & 4.3395e-07 \\

1.8 & 6.519e-6 & 0.45600 & 0.04090 & 7.14e-6  & 4.560e-1 & 4.098e-2 & 1.356e-2 &  6.2158e-07\\

\hline
\end{tabular}
\end{table}
\FloatBarrier
\begin{table}[htbp]
\centering
\caption{Comparison of relative error and computation time for test case 1 with $x\in[0,10]$.}
\label{tab:1D555}
\begin{tabular}{|c|ccccc|c|}
\hline
\multirow{2}{*}{Methods\cite{keshav2025new}} & \multicolumn{5}{c|}{Time} & \multirow{2}{*}{ Time (s)} \\
\cline{2-6}
 & 1.0 & 1.5 & 2.0 & 2.5 & 3.0 &  \\
\hline

HPM   & 4.8030e-05 & 2.6833e-04 & 7.9302e-04 & 1.6617e-03 & 2.8746e-03 & 5.96 \\

BLUES & 4.8030e-05 & 2.6833e-04 & 7.9302e-04 & 1.6617e-03 & 2.8746e-03 & 11.81 \\

APM   & 6.8014e-05 & 7.5897e-05 & 7.6001e-05 & 7.0496e-05 & 6.0488e-05 & 22.78 \\

FEM   & 3.6621e-06 & 5.2734e-06 & 6.5104e-06 & 7.4737e-06 & 8.2397e-06 & 2.51 \\

\hline
\end{tabular}
\end{table}

\textbf{Test Case 2: Product kernel with Dirac delta initial distribution}

For this case, we consider the breakage kernel $
\beta(x,y,z)=\frac{2}{y},
$
together with the collision kernel
$
\Gamma(x,y)=xy.
$
For this problem, the analytical solution is given by
\[
u(x,y,t)
=
e^{-tx}\left[2t + t^2(1 - x)\right] + \delta(x - 1)\, e^{-t}
\]
 To examine the convergence behavior, the grid is successively refined by increasing the number of grid points in each spatial direction at every iteration. The numerical results are summarized in Table~\ref{tab:1D6} with computational domain $
\mathcal{D} := [10^{-9},\,5] 
$. The findings are consistent with the theoretical analysis presented. Figure~\ref{fig:case5_case6} (b) shows the comparison between the exact and numerical solutions obtained. We have compared the $L^1$ error of FEM with the existing scheme on a non-uniform and random grid for t=10 in Table~\ref {tab:case6_1}. The CPU usage time for existing schemes was 78.31s and 83.92s, but the proposed scheme took 47.97s.
\begin{table}[htbp]
\centering
\caption{Error table for NCBE (Test Case 2)}
\label{tab:1D6}
\renewcommand{\arraystretch}{1.2}
\begin{tabular}{|c| c| c| c| c| c| c|}
\hline
$N$ & $\|\mathcal{E}\|_{L^2(\Omega)}$ & EOC 
    & $\mathrm{RelError}_{L^\infty}$ & EOC
    & $\|\mathcal{E}\|_{H^1(\Omega)}$ & EOC \\
\hline
80   & 1.0424e-01 & - & 4.6160e-04 & - & 5.1007e+00 & - \\
160  & 2.9550e-02 & 1.82  & 1.1967e-04 & 1.95 & 2.7061e+00 & 0.91 \\
320  & 7.5001e-03 & 1.98  & 3.0466e-05 & 1.97 & 1.3660e+00 & 0.99 \\
640  & 1.8822e-03 & 1.99  & 7.6859e-06 & 1.99 & 6.8373e-01 & 1.00 \\
1280 & 4.7099e-04 & 2.00 & 1.9302e-06 & 1.99 & 3.4172e-01 & 1.00 \\
\hline
\end{tabular}
\end{table}
\begin{table}[!h]
\centering
\caption{$L^1$ error for non uniform and random grids}
\label{tab:case6_1}
\begin{tabular}{@{}llllll@{}}
\toprule
& \multicolumn{2}{l}{Non uniform} & \phantom{abc} & \multicolumn{2}{l}{Random} \\
\cmidrule{2-3} \cmidrule{5-6}
Grids & Existing \cite{kushwah2023population} & FEM & & Existing & FEM \\
\midrule

60  & 9.02e-02 & 4.0071e-01 & & 0.35   & 4.0228e-01   \\
120 & 3.87e-02 & 1.3689e-01 & &  0.21 & 1.3701e-01\\
240 & 1.77e-02 & 4.0138e-02 & & 0.19 & 4.0135e-02  \\
480 & 8.40e-03 & 1.0877e-02 & & 0.19 & 1.0877e-02  \\
960 & 4.09e-03 & 2.8316e-03 & & 0.19 & 2.8316e-03  \\
\bottomrule
\end{tabular}
\end{table}

\begin{figure}[htbp]
\centering
\begin{subfigure}{0.5\textwidth}
    \centering
    \includegraphics[width=\linewidth]{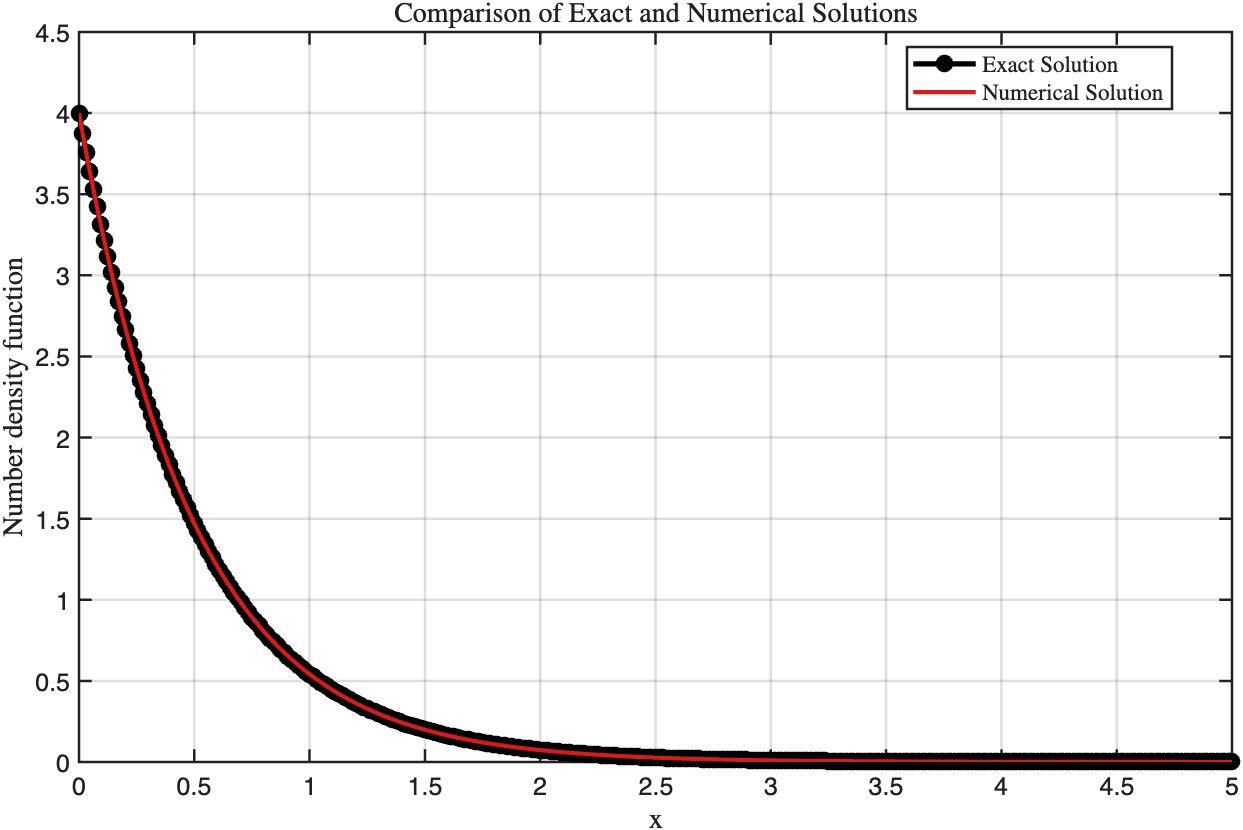}
    \caption{Case 1}
\end{subfigure}
\hfill
\begin{subfigure}{0.45\textwidth}
    \centering
    \includegraphics[width=\linewidth]{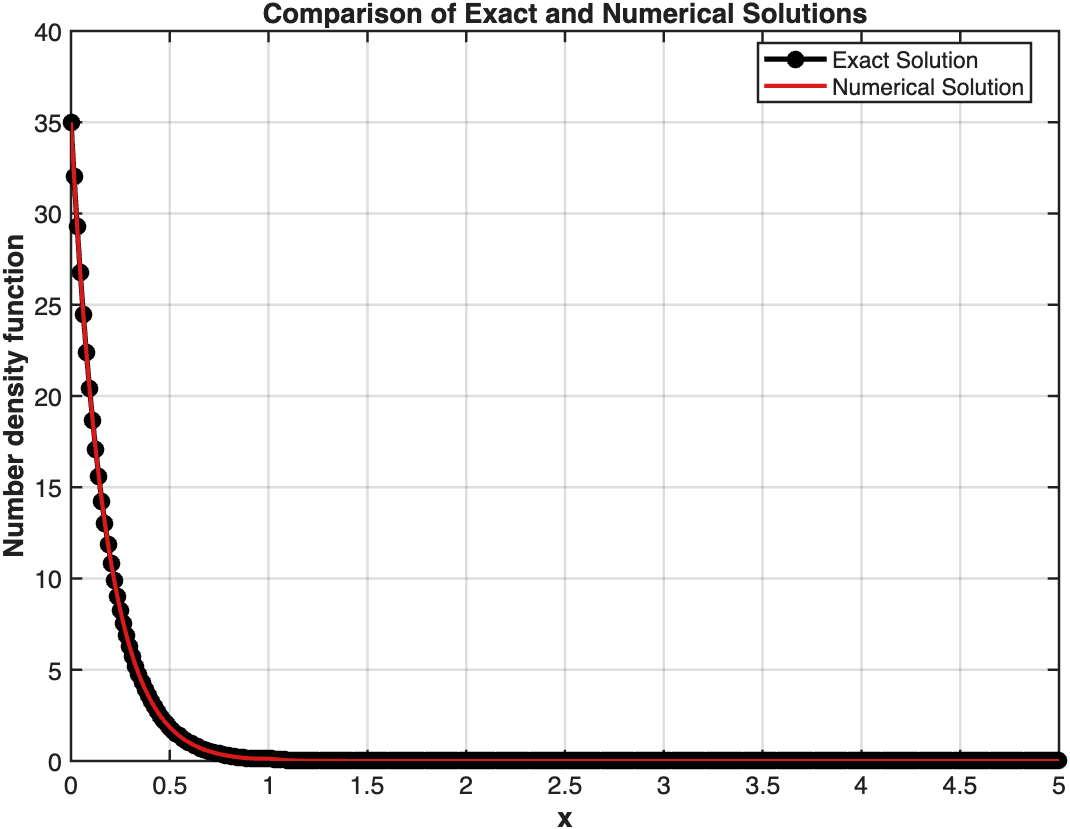}
    \caption{Case 2}
\end{subfigure}
\caption{Comparison of numerical solutions of Test Case 1 and Test Case 2 with the respective exact solutions.}
\label{fig:case5_case6}
\end{figure}

\subsubsection{Two dimensional cases}

\textbf{Test Case 3: Product kernel with exponential initial condition}

In this case, we consider a two-dimensional breakage mechanism in which two distinct properties describe the particles. The corresponding breakage kernels for this test case is $\beta(\mathbf{x},\mathbf{y};\mathbf{z})=\frac{4}{y_1y_2}$ with the product collisional kernel $\Gamma(\mathbf{x},\mathbf{y})=x_1x_2y_1y_2$. Here we consider the  exact solution as $u(\mathbf{x},0)=(1+t)^4exp(-(1+t)(x_1+x_2)) $ .The computational domain is defined as $
\mathcal{D} := [10^{-9},\,2] \times [10^{-9},\,2]$ .To examine convergence behavior, we performed a grid-refinement study by increasing the number of nodes in each spatial dimension. The numerical results are summarized in Table \ref{tab:2D_2}. These findings provide empirical confirmation of the theoretical convergence analysis established in Section~\ref{sec:4}.

\begin{table}[!h]
\centering
\caption{Error table for NCBE (Test Case 3)}
\label{tab:2D_2}
\scriptsize
\setlength{\tabcolsep}{3pt}
\renewcommand{\arraystretch}{1.5}

\resizebox{\textwidth}{!}{%
\begin{tabular}{c cccccc cccccc cccccc}
\toprule
& \multicolumn{6}{c}{$\mathbf{P_1}$}
& \multicolumn{6}{c}{$\mathbf{P_2}$}
& \multicolumn{6}{c}{$\mathbf{P_3}$} \\
\cmidrule(lr){2-7} \cmidrule(lr){8-13} \cmidrule(lr){14-19}

$h$
& $\|E\|_{L^2}$ & EOC & $\|E\|_{H^1}$ & EOC & RelErr & Ord
& $\|E\|_{L^2}$ & EOC & $\|E\|_{H^1}$ & EOC & RelErr & Ord
& $\|E\|_{L^2}$ & EOC & $\|E\|_{H^1}$ & EOC & RelErr & Ord \\
\midrule

1.41421
& 3.48014 & -- & 11.9988 & -- & 0.876857 & --
& 0.573831 & -- & 4.45592 & -- & 0.144583 & --
& 0.262906 & -- & 2.45866 & -- & 0.0662418 & -- \\

0.707107
& 0.996604 & 1.804 & 6.88312 & 0.802 & 0.249286 & 1.815
& 0.0819333 & 2.808 & 1.39383 & 1.677 & 0.0204944 & 2.819
& 0.0197745 & 3.733 & 0.345652 & 2.830 & 0.0049463 & 3.743 \\

0.353553
& 0.257915 & 1.950 & 3.59375 & 0.938 & 0.0645007 & 1.950
& 0.0106545 & 2.943 & 0.371731 & 1.907 & 0.00266453 & 2.943
& 0.00131421 & 3.911 & 0.0375924 & 3.201 & 0.000328663 & 3.912 \\

0.176777
& 0.0648008 & 1.993 & 1.81799 & 0.983 & 0.0162056 & 1.993
& 0.00134638 & 2.984 & 0.0944768 & 1.976 & 0.000336707 & 2.984
& 8.4596e-5 & 3.957 & 0.00417299 & 3.171 & 2.11561e-5 & 3.957 \\

0.0883883
& 0.0161848 & 2.001 & 0.9117 & 0.996 & 0.00404756 & 2.001
& 1.68756e-4 & 2.996 & 0.0237149 & 1.994 & 4.22033e-5 & 2.996
& 5.37065e-6 & 3.977 & 0.000573742 & 2.863 & 1.34311e-6 & 3.977 \\

\bottomrule
\end{tabular}%
}
\end{table}
\subsubsection{Three dimensional cases}
\vspace{-0.3cm}
\textbf{Test Case 4: Product kernel with exponential initial condition}

The corresponding breakage kernels for this test case are $\beta(\mathbf{x},\mathbf{y};\mathbf{z})=\frac{8}{y_1y_2y_3}$ with the product collisional kernel $\Gamma(\mathbf{x},\mathbf{y})=x_1x_2x_3y_1y_2y_3$. Here we consider the  exact solution as $u(\mathbf{x},0)=(1+t)^6exp(-(1+t)(x_1+x_2+x_3)) $ .The computational domain is defined as $
\mathcal{D} := [10^{-9},\,2] \times [10^{-9},\,2]\times [10^{-9},\,2]$ .To examine convergence behavior, we performed a grid-refinement study by increasing the number of nodes in each spatial dimension. The numerical results are summarized in Table \ref{tab:3D_2}. These findings validates the theoretical convergence analysis established in Section~\ref{sec:4}.


\begin{table}[!h]
\centering
\caption{Error table for NCBE (Test Case 4)}
\label{tab:3D_2}
\scriptsize
\setlength{\tabcolsep}{3pt}
\renewcommand{\arraystretch}{1.5}

\resizebox{\textwidth}{!}{%
\begin{tabular}{c cccccc cccccc cccccc}
\toprule
& \multicolumn{6}{c}{$\mathbf{P_1}$}
& \multicolumn{6}{c}{$\mathbf{P_2}$}
& \multicolumn{6}{c}{$\mathbf{P_3}$} \\
\cmidrule(lr){2-7} \cmidrule(lr){8-13} \cmidrule(lr){14-19}

$h$
& $\|E\|_{L^2}$ & EOC & $\|E\|_{H^1}$ & EOC & RelErr & Ord
& $\|E\|_{L^2}$ & EOC & $\|E\|_{H^1}$ & EOC & RelErr & Ord
& $\|E\|_{L^2}$ & EOC & $\|E\|_{H^1}$ & EOC & RelErr & Ord \\
\midrule

3.4641
& 53.3902 & -- & 101.214 & -- & 9.02967 & --
& 21.9696 & -- & 69.2115 & -- & 3.71563 & --
& 9.43861 & -- & 43.3467 & -- & 1.59631 & -- \\

1.73205
& 15.5819 & 1.777 & 53.5353 & 0.919 & 1.99624 & 2.177
& 3.07934 & 2.835 & 22.0405 & 1.651 & 0.394505 & 3.235
& 1.17462 & 3.006 & 10.88 & 1.994 & 0.150484 & 3.407 \\

0.866025
& 4.18123 & 1.898 & 26.8678 & 0.995 & 0.523308 & 1.932
& 0.43134 & 2.836 & 6.97008 & 1.661 & 0.053985 & 2.869
& 0.0963823 & 3.607 & 1.71704 & 2.664 & 0.0120629 & 3.641 \\

0.433013
& 1.06363 & 1.975 & 13.3194 & 1.012 & 0.133023 & 1.976
& 0.0556957 & 2.953 & 1.86878 & 1.899 & 0.00696557 & 2.954
& 0.00661202 & 3.866 & 0.207545 & 3.048 & 0.000826931 & 3.867 \\

\bottomrule
\end{tabular}%
}
\end{table}
\vspace{-0.5cm}

\section{Conclusion}\label{sec:6}
\vspace{-0.35cm}

This study proves that the finite element method (FEM) is a robust and reliable approach for solving the multidimensional nonlinear collisional breakage equation. Due to its inherent simplicity and efficiency, the FEM achieves high accuracy and rapid convergence, which facilitates a deeper understanding of the problem under consideration. The findings confirm the superior performance of the proposed approach by comparing numerical approximations with both exact solutions and available results from other schemes. Even with relatively large time intervals, the approach provides accurate approximations of the number density function. Furthermore, the approach accurately captures the first two moments of the system, which are crucial for determining the total number of particles and the total particle volume, both of which are essential for practical applications. The FEM requires significantly less CPU  time than other available schemes. In the first test case, the approximated solution obtained closely matches the solution available in the literature. To the best of our knowledge, this is one of the first research efforts to provide a detailed error and rate analysis for 2D and 3D cases using manufactured exact solutions. Overall, the proposed method proves that the FEM approach is an exceptionally effective tool, delivering high-quality qualitative and quantitative results. 

\noindent {\bf Data availability:}
Data are available from the authors upon request.

\noindent {\bf Declaration of competing interest:} 
The authors declare no known competing financial interests or personal relationships that could have influenced the work reported in this paper.

\noindent {\bf  Acknowledgments:}  The first author gratefully acknowledges the Ministry of Education for supporting her research.

\bibliographystyle{abbrv}
\bibliography{reference}
\end{document}